\documentclass{siamart1116}

\usepackage{amsmath,amssymb,amstext}

\usepackage{amsfonts}
\usepackage{graphicx}
\usepackage{epstopdf}
\ifpdf
  \DeclareGraphicsExtensions{.eps,.pdf,.png,.jpg}
\else
  \DeclareGraphicsExtensions{.eps}
\fi

\usepackage{mathrsfs}

\usepackage[caption=false]{subfig}

\numberwithin{theorem}{section}

\newcommand{\TheTitle}{A Note on the Quasi-Stationary Distribution\texorpdfstring{\\}{}
of the Shiryaev Martingale on the Positive Half-Line}
\newcommand{\TheAuthors}{A. S. Polunchenko, S. Mart\'{i}nez, and J. San Mart\'{i}n}

\headers{On the Quasi-Stationary Distribution of the Shiryaev Martingale}{\TheAuthors}

\title{{\TheTitle}\thanks{Submitted to the editors DATE.
\funding{The effort of A.~S.~Polunchenko was partially supported by the Simons Foundation via a Collaboration Grant in Mathematics under Award \#\,304574.}}}

\author{
    Aleksey S. Polunchenko%
    \thanks{Department of Mathematical Sciences, State University of New York at Binghamton, Binghamton, New York, USA
    (\email{aleksey@binghamton.edu}, \url{http://people.math.binghamton.edu/aleksey}).}
\and
    Servet Mart\'{i}nez%
    \thanks{CMM-DIM, UMI-CNRS 2807, Basal PBF-03, Universidad de Chile, Santiago, Chile
    (\email{smartine@dim.uchile.cl}, \email{jsanmart@dim.uchile.cl}).}
\and
    Jaime San Mart\'{i}n%
    \footnotemark[3]
}

\usepackage{amsopn}

\renewcommand{\Pr}{\mathbb{P}} % probability
\DeclareMathOperator{\EV}{\mathbb{E}} % expected value

\renewcommand{\le}{\leqslant} % AMS le ge
\renewcommand{\ge}{\geqslant}

\newcommand{\abs}[1]{\left\vert#1\right\vert}

\DeclareMathOperator{\One}{\mathchoice{\rm 1\mskip-4.2mu l}{\rm 1\mskip-4.2mu l}{\rm 1\mskip-4.6mu l}{\rm 1\mskip-5.2mu l}}
\newcommand{\indicator}[1]{{\One_{\left\{#1\right\}}}}

\ifpdf
\usepackage{hyperxmp}

\hypersetup{%
    bookmarks=true,         % show bookmarks bar?
    colorlinks=true,
    citecolor=blue,
    urlcolor=blue,
    unicode=false,          % non-Latin characters in Acrobat’s bookmarks
    pdftoolbar=true,        % show Acrobat’s toolbar?
    pdfmenubar=true,        % show Acrobat’s menu?
    pdffitwindow=false,     % window fit to page when opened
    pdfstartview={FitH},    % fits the width of the page to the window
    pdftitle={A Note on the Quasi-Stationary Distribution of the Shiryaev Martingale on the Positive Half-Line},    % title
    pdfauthor={Aleksey S. Polunchenko and Servet Mart\'{i}nez and Jaime San Mart\'{i}n},     % author
    pdfsubject={A Note on the Quasi-Stationary Distribution of the Shiryaev Martingale on the Positive Half-Line},   % subject of the document
    pdfcreator={Aleksey S. Polunchenko},   % creator of the document
    pdfproducer={MikTeX}, % producer of the document
    pdfkeywords={Quasi-stationary distribution,Shiryaev martingale,Whittaker functions}, % list of keywords
    pdfnewwindow=true,      % links in new window
    linkcolor=red,          % color of internal links (change box color with linkbordercolor)
    filecolor=magenta      % color of file links
}
\fi

\graphicspath{{./gfx/}}

\usepackage[norefs,nocites]{refcheck}
\usepackage{silence}
\begin{document}

\maketitle

% REQUIRED
\begin{abstract}
We obtain a closed-form formula for the quasi-stationary distribution of the classical Shiryaev martingale diffusion considered on the positive half-line  $[A,+\infty)$ with $A>0$ fixed; the state space's left endpoint is assumed to be the killing boundary. The formula is obtained analytically as the solution of the appropriate singular Sturm--Liouville problem; the latter was first considered in~\cite[Section~7.8.2]{Collet+etal:Book2013}, but has heretofore remained unsolved.
\end{abstract}

% REQUIRED

% REQUIRED

%+-----------------------------------------------------------------------------------------------+%
\section{Introduction and problem formulation}
\label{sec:intro}

This work centers around the stochastic process known as the Shiryaev martingale. Specifically, the latter is defined as the solution $(X_{t})_{t\ge0}$ of the stochastic differential equation (SDE)
\begin{equation}\label{eq:ShiryaevMartingale-def}
dX_t
=
dt+ X_t\,dB_t
\;\;\text{with}\;\;
X_0\triangleq x_0\ge0\;\;\text{fixed},
\end{equation}
where $(B_t)_{t\ge0}$ is standard Brownian motion (i.e., $\EV[dB_t]=0$, $\EV[(dB_t)^2]=dt$, and $B_0=0$); the initial value $X_0=x_0\ge0$ is sometimes also referred to as the headstart. The name ``Shiryaev martingale'' has apparently been introduced in~\cite[Section~7.8.2]{Collet+etal:Book2013} and is due to two reasons. The first reason is to acknowledge the fact that equation~\eqref{eq:ShiryaevMartingale-def} was first arrived at and extensively studied by Prof.~A.N.~Shiryaev in his fundamental work (see~\cite{Shiryaev:SMD61,Shiryaev:TPA63}) in the area of quickest change-point detection where the process $(X_t)_{t\ge0}$ has become known as the Shiryaev--Roberts detection statistic, and it is one of the ``central threads'' in the field. See also, e.g.,~\cite{Pollak+Siegmund:B85,Shiryaev:Bachelier2002,Feinberg+Shiryaev:SD2006,Burnaev+etal:TPA2009,Polunchenko:SA2017,Polunchenko:TPA2017}. The second reason is that the time-homogeneous Markov diffusion $(X_t)_{t\ge0}$ is easy to see to have the martingale property $\EV[X_t-x_0-t]=0$ for any $t\ge0$, i.e., the process $\{X_t-x_0-t\}_{t\ge0}$ is a zero-mean martingale.

The SDE~\eqref{eq:ShiryaevMartingale-def} is a special case of the more general SDE $dZ_t=(aZ_t+b)\,dt+Z_t\,dB_t$ where $Z_0=z_0\in\mathbb{R}$, $a\in\mathbb{R}$, and $b\in\mathbb{R}$ are fixed; note that the process $\{Z_t-z_0-t\}_{t\ge0}$ is not a zero-mean martingale, unless $a=0$ and $b=1$. The process $(Z_t)_{t\ge0}$ is sometimes called the Shiryaev process or the Shiryaev diffusion, for it, too, was arrived at and studied by Prof.~A.N.~Shiryaev in~\cite{Shiryaev:SMD61,Shiryaev:TPA63} in the context of quickest change-point detection. However, the Kolmogorov~\cite{Kolmogorov:MA1931} forward and backward equations corresponding to the Shiryaev process $(Z_t)_{t\ge0}$ have also arisen independently in areas far beyond quickest change-point detection, notably in mathematical physics, and, more recently, in mathematical finance. By way of example, in mathematical physics, the authors of~\cite{Monthus+Comtet:JPhIF1994,Comtet+Monthus:JPhA1996} dealt with $Z_t$ interpreting it as the position at time $t$ of a particle moving around in an inhomogeneous environment driven by a combination of random forces (e.g., thermal noise). Financial significance of the Shiryaev process has been understood, e.g., in~\cite{Geman+Yor:MF1993,Donati-Martin+etal:RMI2001,Linetsky:OR2004}, in relation to so-called arithmetic Asian options where $Z_t$ represents the option's price at time $t$. Moreover, the Shiryaev process also proved useful as a stochastic interest rate model. See, e.g.,~\cite{Geman+Yor:MF1993,Vanneste+etal:IME1994,DeSchepper+etal:IME1994,Milevsky:IME1997,DeSchepper+Goovaerts:IME1999}. Last but not least, the Shiryaev diffusion is also of interest in itself as a stochastic process, especially due to its close connection to geometric Brownian motion (which $Z_t$ is when $b=0$). See, e.g.,~\cite{Wong:SPMPE1964,Yor:AAP1992,Donati-Martin+etal:RMI2001,Dufresne:AAP2001,Schroder:AAP2003,Peskir:Shiryaev2006,Polunchenko+Sokolov:MCAP2016}.

The Shiryaev process---whether $(X_t)_{t\ge0}$ or $(Z_t)_{t\ge0}$---is typically considered either on a compact subset of the real line, or on the entire real line, or on one of the two half-lines $(-\infty,0]$ or $[0,+\infty)$. This work draws attention to the case when the state space for $(X_t)_{t\ge0}$ is the set $[A,+\infty)$ where $A>0$ is given. Specifically, we shall assume that $(X_t)_{t\ge0}$ is started off a point inside the interval $[A,+\infty)$, i.e., $X_0\triangleq x_0\in[A,+\infty)$, and then the process is let continue until it hits the lower boundary $A>0$ whereat the process is terminated. The question of interest is the process' long-term behavior, conditional that the process is not killed. More formally, consider the stopping time
\begin{equation}\label{eq:T-def}
S_A
\triangleq
\inf\{t\ge0\colon X_{t}=A\}\;\;\text{such that}\;\;\inf\{\varnothing\}=+\infty,
\end{equation}
where $A>0$ is given. The specific aim of this paper is to obtain an exact closed-form formula for the Shiryaev martingale's quasi-stationary distribution. Specifically, this distribution is defined as
\begin{equation}\label{eq:QSD-def}
Q_A(x)
\triangleq
\lim_{t\to+\infty}\Pr(X_t\le x|{S}_A>t)
\;\;\text{with}\;\;
q_A(x)
\triangleq
\dfrac{d}{dx}Q_A(x),
\;\text{where}\;
x\in[A,+\infty),
\end{equation}
and $X_0\triangleq x_0>A$ and $A>0$ are preset. The existence of this distribution has been previously asserted in~\cite[Section~7.8.2]{Collet+etal:Book2013}, which, to the best of our knowledge, is also where the very problem of finding either $Q_A(x)$ or $q_A(x)$ in a closed form was first formulated, but has heretofore remained unsolved. The solution we obtain analytically in Section~\ref{sec:qsd-solution} elucidates the general theory of quasi-stationary phenomena associated with killed one-dimensional Markov diffusions set forth in the seminal work of Mandl~\cite{Mandl:CMJ1961} and then further developed by Collet, Mart\'{i}nez and San Mart\'{i}n in~\cite{Collet+etal:AP1995,Martinez+SanMartin:JTP2001,Martinez+SanMartin:AP2004,Collet+etal:Book2013}; see also~\cite{Cattiaux+etal:AP2009}. The obtained formulae for $Q_A(x)$ or $q_A(x)$ complement those previously found by Polunchenko~\cite{Polunchenko:SA2017} for the case when the Shiryaev martingale is restricted to the interval $[0,A]$ with $A>0$ fixed.

%+-----------------------------------------------------------------------------------------------+%
\section{Preliminaries}
\label{sec:preliminaries}

For notational brevity, we shall henceforth omit the subscript ``$A$'' in ``$Q_A(x)$'' as well as in ``$q_A(x)$'', unless the dependence on $A$ is noteworthy. Also, for technical convenience and without loss of generality, we shall primarily deal with $q(x)$ rather than with $Q(x)$.

It is has already been established in the literature (see, e.g.,~\cite[Section~7.8.2]{Collet+etal:Book2013} or~\cite{Mandl:CMJ1961,Cattiaux+etal:AP2009}) that $q(x)$, formally defined in~\eqref{eq:QSD-def}, is the solution of a certain boundary-value problem composed of a second-order ordinary differential equation (ODE) considered on the half line $(A,+\infty)$, a normalization constraint, a set of boundary conditions along with a square-integrability restriction. Specifically, the ODE---which we shall refer to as the {\em master equation}---is of the form
\begin{equation}\label{eq:master-eqn}
\dfrac{1}{2}\dfrac{d^2}{dx^2}\big[x^2\,q(x)\big]
-
\dfrac{d}{dx}\big[q(x)\big]
=
-\lambda\,q(x),
\;\;x\in(A,+\infty),
\end{equation}
where $\lambda\ge0$ is the {\em smallest} eigenvalue of the differential operator
\begin{equation}\label{eq:Doperator-def}
\mathscr{D}
\triangleq
\dfrac{1}{2}\dfrac{\partial^2}{\partial x^2} x^2
-
\dfrac{\partial}{\partial x},
\end{equation}
which is the infinitesimal generator of the Shiryaev diffusion $(X_t)_{t\ge0}$ governed by SDE~\eqref{eq:ShiryaevMartingale-def}; we remark that the nonnegativity of $\lambda$ is {\em not} an assumption, and it will be formally asserted below that, in fact, $\lambda\in(0,1/8]$. It goes without saying that $\lambda$ is dependent on $A$, and, wherever necessary, we shall emphasize this dependence via the notation $\lambda_{A}$.

The relation
\begin{equation}\label{eq:lambdaA-kill-rate}
\lambda_{A}
=
-\lim_{t\to+\infty}\left\{\dfrac{1}{t}\Pr(S_A>t)\right\},
\end{equation}
where $S_A$ is the stopping time defined in~\eqref{eq:T-def}, lends $\lambda_A$ the flavor of the killing rate for $(X_t)_{t\ge0}$; cf.~\cite{Collet+etal:Book2013}. In fact, in~\cite{Collet+etal:Book2013}, the relation~\eqref{eq:lambdaA-kill-rate} is used to argue that $\lambda_A$ is an increasing function of $A>0$. Moreover, it is also established in~\cite[Section~7.8.2]{Collet+etal:Book2013} that $\lambda_A=1/8$ is guaranteed for $A\ge e^3\approx 20.0855369$. We shall see in the next section that $\lambda_A=1/8$ is attained for much smaller values of $A$, namely for $A\ge A^{*}\approx1.265857361$ with $A^{*}$ being the solution of a certain transcendental equation.

The normalization constraint that $q(x)$ is to satisfy is the natural requirement
\begin{align}\label{eq:QSD-pdf-norm-constraint}
\int_{A}^{+\infty}q(x)\,dx
&=
1,
\end{align}
which is merely the statement that $q(x)$, as a pdf supported on $[A,+\infty)$, must integrate to unity over its entire support.

It is easily checked that the state space's lower boundary $x=A>0$ is a regular absorbing boundary, and the upper boundary $x=+\infty$ is a natural boundary. Therefore, there is only one boundary condition to impose on $q(x)$, and this condition is at $x=A$, and it is as follows:
\begin{equation}\label{eq:QSD-pdf-bnd-cond-A}
q(A)
=
0,
\end{equation}
which, in ``differential equations speak'', is a Dirichlet-type boundary condition. While no boundary condition is required at $x=+\infty$, there is a certain square-integrability restriction required to hold for $q(x)$ around $x=+\infty$. This restriction will be explained below.

Subject to the absorbing boundary condition~\eqref{eq:QSD-pdf-bnd-cond-A}, the normalization constraint~\eqref{eq:QSD-pdf-norm-constraint}, and the square-integrability restriction yet to be discussed, the master equation~\eqref{eq:master-eqn} is a Sturm--Liouville problem. It is a singular problem, for the domain, i.e., the interval $[A,+\infty)$, is unbounded. The singular nature of the problem affects the spectrum $\{\lambda\}$ of the corresponding operator $\mathscr{D}$ given by~\eqref{eq:Doperator-def}. By virtue of the multiplying factor
\begin{equation}\label{eq:speed-measure}
\mathfrak{m}(x)
\triangleq
\dfrac{2}{x^2}\,e^{-\tfrac{2}{x}}
\end{equation}
the master equation can be brought to the canonical Sturm--Liouville form
\begin{equation}\label{eq:eigen-eqn-two}
\dfrac{1}{2}\dfrac{d}{dx}\left[x^2\,\mathfrak{m}(x)\,\dfrac{d}{dx}\varphi(x)\right]
=
-\lambda\,\mathfrak{m}(x)\,\varphi(x),
\end{equation}
where the unknown function $\varphi(x)$ is such that $q(x)\propto\mathfrak{m}(x)\,\varphi(x)$, i.e., $q(x)$ is a multiple of $\mathfrak{m}(x)\,\varphi(x)$. Hence, the operator $\mathscr{D}$ given by~\eqref{eq:Doperator-def} is equivalent to the Sturm--Liouville operator
\begin{equation}\label{eq:Goperator-def}
\mathscr{G}
\triangleq
\dfrac{1}{2\,\mathfrak{m}(x)}\dfrac{\partial}{\partial x}x^2\,\mathfrak{m}(x)\,\dfrac{\partial}{\partial x}
\end{equation}
with $\mathfrak{m}(x)$ given by~\eqref{eq:speed-measure}. Although operators $\mathscr{G}$ and $\mathscr{D}$ are essentially formal adjoints of one another, the former is more convenient to deal with because it is in a canonical Sturm--Liouville form, so that the general theory of Sturm--Liouville operators can be readily utilized to gain preliminary insight into the spectral characteristics of $\mathscr{G}$. It is evident from~\eqref{eq:Doperator-def} and~\eqref{eq:Goperator-def} that $\mathscr{G}$ and $\mathscr{D}$ have the same spectra, and that their corresponding eigenfunctions differ by a factor of $\mathfrak{m}(x)$ given by~\eqref{eq:speed-measure}.

The general theory of second-order differential operators or Sturm--Liouville operators (such as our operators $\mathscr{G}$ and $\mathscr{D}$ introduced above) is well-developed, and, in particular, the spectral properties of such operators are well-understood. The classical fundamental references on the subject are~\cite{Titchmarsh:Book1962},~\cite{Levitan:Book1950},~\cite{Coddington+Levinson:Book1955},~\cite{Dunford+Schwartz:Book1963}, and~\cite{Levitan+Sargsjan:Book1975}; for applications of the theory to diffusion processes, see, e.g., Ito and McKean~\cite[Section~4.11]{Ito+McKean:Book1974}, and especially Linetsky~\cite{Linetsky:HandbookChapter2007} who provides a great overview of the state-of-the-art in the field considered in the context of stochastic processes. For our specific problem, the general Sturm--Liouville theory immediately establishes that the eigenfunctions $\varphi(x)\equiv\varphi(x,\lambda)$ indexed by $\lambda$ of the operator $\mathscr{G}$ given~\eqref{eq:Goperator-def} form an orthonormal basis in the Hilbert space $L^2([A,+\infty),\mathfrak{m})$ of real-valued $\mathfrak{m}(x)$-measurable, square-integrable (with respect to the $\mathfrak{m}(x)$ measure) functions defined on the interval $[A,+\infty)$ equipped with the ``$\mathfrak{m}(x)$''-weighted inner product:
\begin{equation*}
\langle f,g\rangle_{\mathfrak{m}}
\triangleq
\int_{A}^{+\infty}\mathfrak{m}(x)\,f(x)\,g(x)\,dx.
\end{equation*}

More specifically, the foregoing means that if $\lambda^{(i)}$ and $\lambda^{(j)}$ are any two eigenvalues of $\mathscr{G}$, and $\varphi(x,\lambda^{(i)})$ and $\varphi(x,\lambda^{(j)})$ are the corresponding eigenfunctions, then
\begin{equation*}%\label{eq:eigenfun-orthonorm}
\int_{A}^{+\infty}\mathfrak{m}(x)\,\varphi(x,\lambda^{(i)})\,\varphi(x,\lambda^{(j)})\,dx
=
\indicator{i=j},
\end{equation*}
where $\varphi(x,\lambda^{(i)})$ and $\varphi(x,\lambda^{(j)})$ are each assumed to be of unit ``length'', in the sense that $\|\varphi(\cdot,\lambda^{(i)})\|=1=\|\varphi(\cdot,\lambda^{(j)})\|$, with the ``length'' defined as
\begin{equation}\label{eq:eigfun-norm-def}
\|\varphi(\cdot,\lambda)\|^2
\triangleq
\int_{A}^{+\infty}\mathfrak{m}(x)\,\varphi^2(x,\lambda)\,dx.
\end{equation}

We are now in a position to state the square-integrability restriction on $q(x)$: it is the requirement that $\|\varphi(\cdot,\lambda)\|<+\infty$ for the very same $\lambda$ that is present in~\eqref{eq:master-eqn}.

To gain further insight into the spectral properties of the operator $\mathscr{G}$ we turn to the work of Linetsky~\cite{Linetsky:HandbookChapter2007} who introduces three mutually exclusive Spectral Categories of Sturm--Liouville operators, and establishes easy-to-use criteria to determine which specific category a given Sturm--Liouville operator falls under. The classification is based on the nature of the corresponding domain boundaries, viz. whether the boundaries are oscillatory or non-oscillatory. The classification criteria are given by~\cite[Theorem~3.3,~p.~248]{Linetsky:HandbookChapter2007}. Specifically, by appealing to~\cite[Theorem~3.3(ii),~p.~248]{Linetsky:HandbookChapter2007} it is straightforward to verify that our Sturm--Liouville problem belongs to Spectral Category II introduced in~\cite[Theorem~3.2(ii),~p.~246]{Linetsky:HandbookChapter2007}. This means that the spectrum $\{\lambda\}$ of the operators $\mathscr{G}$ and $\mathscr{D}$ is simple and nonnegative, i.e., all $\lambda\ge0$. Moreover, the spectrum is purely absolutely continuous in $(1/8,+\infty)$ where the $1/8$ is the spectrum cutoff point. Finally, since~\cite[Theorem~3.3(ii),~p.~248]{Linetsky:HandbookChapter2007} shows that $x=+\infty$ is a non-oscillatory boundary, the operators $\mathscr{G}$ and $\mathscr{D}$ may also have a finite set of simple eigenvalues inside the interval $[0,1/8]$, and these eigenvalues are determined entirely by the Dirichlet boundary condition~\eqref{eq:QSD-pdf-bnd-cond-A} or equivalently $\varphi(A,\lambda)=0$.

We conclude this section with a remark that $\lambda=0$ is not an option. Indeed, observe that in this case the function $\varphi(x,0)=K$ for some constant $K$ does solve~\eqref{eq:eigen-eqn-two} and is square-integrable with respect to the $\mathfrak{m}(x)$ measure given by~\eqref{eq:speed-measure}. However, the function $q(x)\propto \mathfrak{m}(x)$, although does solve~\eqref{eq:master-eqn}, is possible to normalize in accordance with the normalization constraint~\eqref{eq:QSD-pdf-norm-constraint} only if $K\neq0$. Yet, if $K\neq0$, then the absorbing boundary condition~\eqref{eq:QSD-pdf-bnd-cond-A} is impossible to fulfil, because, in view of~\eqref{eq:speed-measure}, no nontrivial multiple of $\mathfrak{m}(x)$ can be turned into zero at any finite $x=A>0$. On the other hand, if $K=0$, then the absorbing boundary condition~\eqref{eq:QSD-pdf-bnd-cond-A} is trivially satisfied, but the normalization constraint~\eqref{eq:QSD-pdf-norm-constraint} can never be. Therefore $\lambda=0$ is not an eigenvalue of the operator $\mathscr{D}$ given by~\eqref{eq:Doperator-def}, and its spectrum $\{\lambda\}$ lies entirely inside the interval $(0,1/8]$. This was previously conjectured in~\cite[Section~7.8.2]{Collet+etal:Book2013}. The strict positivity of the smallest eigenvalue of $\mathscr{D}$ enables us to enjoy all of the results already obtained in~\cite[Section~7.8.2]{Collet+etal:Book2013}, starting from the very fact that $q(x)$ must exist. In the next section this distribution will be expressed in a closed form through analytic solution of the corresponding Sturm--Liouville problem.

%+-----------------------------------------------------------------------------------------------+%
\section{The quasi-stationary distribution formulae}
\label{sec:qsd-solution}

The plan now is to fix $A>0$ and solve the master equation~\eqref{eq:master-eqn} analytically and thereby recover both $q_A(x)$ and $Q_A(x)$ in a closed form for all $x\in[A,+\infty)$. To that end, it is easier to deal with the equivalent equation~\eqref{eq:eigen-eqn-two}, and the first step to treat it is to apply the change of variables
\begin{equation}\label{eq:ux-var-def}
x\mapsto u=u(x)=\dfrac{2}{x},\;\text{so that}\;u\mapsto x=x(u)=\dfrac{2}{u}\;\text{and}\;\dfrac{dx}{x}=-\dfrac{du}{u},
\end{equation}
along with the substitution
\begin{equation}\label{eq:eigfun-substitution}
\varphi(x)\mapsto\varphi(u)\triangleq\dfrac{v(u)}{\sqrt{\mathfrak{m}(u)}}\propto \dfrac{1}{u}\,e^{\tfrac{u}{2}}\,v(u),
\end{equation}
to bring the equation to the form
\begin{equation}\label{eq:eigenfcn-eqn-whit-form}
v_{uu}(u)+\left\{-\frac{1}{4}+\frac{1}{u}+\frac{1/4-\xi^2/4}{u^2}\right\}v(u)
=
0,
\end{equation}
where
\begin{equation}\label{eq:xi-def}
\xi
\equiv
\xi(\lambda)
\triangleq
\sqrt{1-8\lambda}
\;\;\text{so that}\;\;
\lambda
\equiv
\lambda(\xi)
=
\dfrac{1}{8}(1-\xi^2),
\end{equation}
and $\xi\in[0,1)$ on account of $\lambda\in(0,1/8]$ concluded in the previous section. The restriction $\xi\in[0,1)$ will be invoked repeatedly throughout the remainder of this section. The change of variables~\eqref{eq:ux-var-def} and the substitution~\eqref{eq:eigfun-substitution} were devised to treat a similar Sturm--Liouville problem in the closely related work of Polunchenko~\cite{Polunchenko:SA2017}; see also~\cite{Linetsky:OR2004} and~\cite{Polunchenko+Sokolov:MCAP2016}. We also remark that equation~\eqref{eq:eigenfcn-eqn-whit-form} is symmetric with respect to the sign of $\xi$, i.e., one could also define $\xi$ as $\xi\triangleq-\sqrt{1-8\lambda}$. However, as will become clear shortly, this ambiguity in the definition of $\xi$ has no effect on the sought quasi-stationary density $q(x)$ whatsoever.

The obtained equation~\eqref{eq:eigenfcn-eqn-whit-form} is a particular case of the classical Whittaker~\cite{Whittaker:BAMS1904} equation
\begin{equation}\label{eq:Whittaker-eqn}
w_{zz}(z)+\left\{-\dfrac{1}{4}+\dfrac{a}{z}+\dfrac{1/4-b^2}{z^2}\right\}w(z)
=
0,
\end{equation}
where $w(z)$ is the unknown function of $z\in\mathbb{C}$, and $a,b\in\mathbb{C}$ are specified parameters. This self-adjoint homogeneous second-order ODE is well-known in mathematical physics as well as in mathematical finance. Its two linearly independent fundamental solutions are known as the Whittaker functions. These functions are special functions that take a variety of forms depending on the specific values of $a$ and $b$. The classical references on the general theory of the Whittaker equation~\eqref{eq:Whittaker-eqn} and Whittaker functions are~\cite{Slater:Book1960} and~\cite{Buchholz:Book1969}. For our purposes it will prove convenient and sufficient to deal with the Whittaker $M$ and $W$ functions, which are conventionally denoted, respectively, as $M_{a,b}(z)$ and $W_{a,b}(z)$, where the indices $a$ and $b$ are the Whittaker's equation~\eqref{eq:Whittaker-eqn} parameters.

By combining~\eqref{eq:Whittaker-eqn},~\eqref{eq:eigenfcn-eqn-whit-form},~\eqref{eq:eigfun-substitution} and~\eqref{eq:ux-var-def}, one can now see that the {\em general} form of $\varphi(x,\lambda)$ is
\begin{equation}\label{eq:eigfun-gen-form}
\varphi(x,\lambda)
=
x\,e^{\tfrac{1}{x}}\left\{B_1\,M_{1,\tfrac{1}{2}\xi(\lambda)}\left(\dfrac{2}{x}\right)+B_2\,W_{1,\tfrac{1}{2}\xi(\lambda)}\left(\dfrac{2}{x}\right)\right\},
\;\;
x\in[A,+\infty),
\end{equation}
where $\xi(\lambda)$ is as in~\eqref{eq:xi-def}, and $B_1$ and $B_2$ are arbitrary constants such that $B_1B_2\neq0$. It is of note that both of the two Whittaker functions involved in the obtained formula for $\varphi(x,\lambda)$ are well-defined, real-valued, and linearly independent of each other for any $\xi(\lambda)\in[0,1)$. Also, from~\cite[Identity~13.1.34,~p.~505]{Abramowitz+Stegun:Handbook1964}, i.e., from the identity
\begin{equation}\label{eq:WhitWM-identity}
W_{a,b}(z)
=
\dfrac{\Gamma(-2b)}{\Gamma(1/2-b-a)}M_{a,b}(z)+\dfrac{\Gamma(2b)}{\Gamma(1/2+b-a)}M_{a,-b}(z),
\end{equation}
where here and onward $\Gamma(z)$ denotes the Gamma function (see, e.g.,~\cite[Chapter~6]{Abramowitz+Stegun:Handbook1964}), one can readily conclude that~\eqref{eq:eigfun-gen-form} is unaffected by the sign ambiguity in the definition~\eqref{eq:xi-def} of $\xi\equiv\xi(\lambda)$.

The obvious next step is to recall that $q(x)\propto \mathfrak{m}(x)\,\varphi(x,\lambda)$, where $\mathfrak{m}(x)$ is given by~\eqref{eq:speed-measure}, and, in view of~\eqref{eq:eigfun-gen-form}, conclude that the quasi-stationary density $q(x)$ has the general form
\begin{equation}\label{eq:QSD-pdf-gen-form}
q(x)
=
\dfrac{1}{x}\,e^{-\tfrac{1}{x}}\left\{C_1\,M_{1,\tfrac{1}{2}\xi(\lambda)}\left(\dfrac{2}{x}\right)+C_2\,W_{1,\tfrac{1}{2}\xi(\lambda)}\left(\dfrac{2}{x}\right)\right\},
\;\;
x\in[A,+\infty),
\end{equation}
where $C_1$ and $C_2$ are constant factors to be designed so as to make $q(x)$ satisfy the absorbing boundary condition~\eqref{eq:QSD-pdf-bnd-cond-A} as well as the normalization constraint~\eqref{eq:QSD-pdf-norm-constraint}. With regard to the former, it is straightforward to see from~\eqref{eq:QSD-pdf-gen-form} that $C_1$ and $C_2$ must satisfy the equation
\begin{equation}\label{eq:C1C2-eqn1}
C_1\,M_{1,\tfrac{1}{2}\xi(\lambda)}\left(\dfrac{2}{A}\right)+C_2\,W_{1,\tfrac{1}{2}\xi(\lambda)}\left(\dfrac{2}{A}\right)
=
0,
\end{equation}
and it is not a degenerate equation in the sense that the Whittaker $M$ and $W$ involved in it are never zeros simultaneously, for the Whittaker $M$ and $W$ functions with the same indices and arguments (finite) may become zeros at the same time only at the origin.

To proceed, observe that
\begin{equation}\label{eq:WhitW-near-zero-asymptotics}
\begin{split}
W_{a,b}(x)
&\sim
\dfrac{\Gamma(-2b)}{\Gamma(1/2-b-a)}\,x^{\tfrac{1}{2}+b}\,e^{-\tfrac{1}{2}x}+\\
&\qquad\qquad\qquad\qquad
+\dfrac{\Gamma(2b)}{\Gamma(1/2+b-a)}\,x^{\tfrac{1}{2}-b}\,e^{-\tfrac{1}{2}x}
\;\;\text{as}\;\; x\to0+,
\end{split}
\end{equation}
which is a direct consequence of~\eqref{eq:WhitWM-identity} and the asymptotics
\begin{equation}\label{eq:WhitM-near-zero-asymptotics}
M_{a,b}(x)
\sim
x^{\tfrac{1}{2}+b}\,e^{-\tfrac{1}{2}x}
\;\;\text{as}\;\; x\to0+,
\end{equation}
established, e.g., in~\cite[Section~16.1,~p.~337]{Whittaker+Watson:Book1927}. Recalling yet again that $\xi\in[0,1)$, it follows from~\eqref{eq:QSD-pdf-gen-form} and the asymptotics~\eqref{eq:WhitW-near-zero-asymptotics} and~\eqref{eq:WhitM-near-zero-asymptotics} that $q(x)$ is $\mathfrak{m}(x)$-measurable for any $\lambda\in(0,1/8]$. Hence, let us fix $\lambda\in(0,1/8]$ and attempt to normalize $q(x)$ given by~\eqref{eq:QSD-pdf-gen-form} in accordance with the normalization constraint~\eqref{eq:QSD-pdf-norm-constraint}. To do so, we turn to~\cite[Integral~7.623.3,~p.~832]{Gradshteyn+Ryzhik:Book2014} which states that
\begin{equation}\label{eq:WhitM-int-formula}
\begin{split}
\int_{0}^{t}(t-x)^{c-1}\,x^{a-1}&\, e^{-\tfrac{1}{2}x}\,M_{a+c,b}(x)\,dx
=\\
&\qquad=
\dfrac{\Gamma(c)\Gamma(a+b+1/2)}{\Gamma(a+b+c+1/2)}\,t^{a+c-1}\,e^{-\tfrac{1}{2}t}\,M_{a,b}(t),\\
&\qquad\qquad\qquad
\;\text{provided $\Re(a+b)>-1/2$ and $\Re(c)>0$},
\end{split}
\end{equation}
and to~\cite[Integral~7.623.8,~p.~833]{Gradshteyn+Ryzhik:Book2014} which states that
\begin{equation}\label{eq:WhitW-int-formula}
\begin{split}
\int_{0}^{1}&(1-x)^{c-1}\,x^{a-c-1}\, e^{-\tfrac{1}{2}tx}\,W_{a,b}(tx)\,dx
=\\
&\qquad=
\Gamma(c)\,e^{-\tfrac{t}{2}}\sec[(a-c-b)\pi]\Biggl\{\sin(c\pi)\,\dfrac{\Gamma(a-c+b+1/2)}{\Gamma(2b+1)}\,M_{a-c,b}(t)+\\
&\qquad\qquad\qquad\qquad+\cos[(a-b)\pi]\,W_{a-c,b}(t)\Biggr\},\\
&\qquad\qquad\qquad
\;\text{provided $0<\Re(c)<\Re(a)-\abs{\Re(b)}+1/2$},
\end{split}
\end{equation}
and obtain
\begin{equation*}
\int_{A}^{+\infty}q(x)\,dx
=
e^{-\tfrac{1}{A}}\left\{C_1\,\dfrac{2}{\xi(\lambda)+1}\,M_{0,\tfrac{1}{2}\xi(\lambda)}\left(\dfrac{2}{A}\right)-C_2\,W_{0,\tfrac{1}{2}\xi(\lambda)}\left(\dfrac{2}{A}\right)\right\},
\end{equation*}
whence
\begin{equation}\label{eq:C1C2-eqn2}
C_1\,\dfrac{2}{\xi(\lambda)+1}\,M_{0,\tfrac{1}{2}\xi(\lambda)}\left(\dfrac{2}{A}\right)-C_2\,W_{0,\tfrac{1}{2}\xi(\lambda)}\left(\dfrac{2}{A}\right)
=
e^{\tfrac{1}{A}},
\end{equation}
which is another equation that the constants $C_1$ and $C_2$ involved in~\eqref{eq:QSD-pdf-gen-form} are to satisfy, and, just as~\eqref{eq:C1C2-eqn1}, this equation is also nondegenerate. Therefore, by solving equations~\eqref{eq:C1C2-eqn1} and~\eqref{eq:C1C2-eqn2} for $C_1$ and $C_2$, and then plugging them over into~\eqref{eq:QSD-pdf-gen-form}, we arrive at the formula
\begin{equation}\label{eq:QSD-pdf-gen-answer0}
\begin{split}
q(x)
&=
\dfrac{C}{x}\,e^{-\tfrac{1}{x}}\Biggl\{W_{1,\tfrac{1}{2}\xi(\lambda)}\left(\dfrac{2}{A}\right)M_{1,\tfrac{1}{2}\xi(\lambda)}\left(\dfrac{2}{x}\right)-\\
&\qquad\qquad\qquad\qquad
-M_{1,\tfrac{1}{2}\xi(\lambda)}\left(\dfrac{2}{A}\right)W_{1,\tfrac{1}{2}\xi(\lambda)}\left(\dfrac{2}{x}\right)\Biggr\}\,\indicator{x\in[A,+\infty)},
\end{split}
\end{equation}
where $A>0$ and $\xi(\lambda)$ is as in~\eqref{eq:xi-def} with $\lambda\in(0,1/8]$ arbitrary, and $C$ is the normalizing factor given by
\begin{equation}\label{eq:norm-factor-C-answer1}
\begin{split}
C
&\triangleq
\left.e^{\tfrac{1}{A}}\right/\Biggl\{W_{0,\tfrac{1}{2}\xi(\lambda)}\left(\dfrac{2}{A}\right)M_{1,\tfrac{1}{2}\xi(\lambda)}\left(\dfrac{2}{A}\right)+\\
&\qquad\qquad\qquad\qquad
+\dfrac{2}{\xi(\lambda)+1}\,M_{0,\tfrac{1}{2}\xi(\lambda)}\left(\dfrac{2}{A}\right)W_{1,\tfrac{1}{2}\xi(\lambda)}\left(\dfrac{2}{A}\right)\Biggr\}.
\end{split}
\end{equation}

It turns out that the above formula for $C$ can be substantially simplified with the aid of the identity
\begin{equation*}
W_{0,b}(z)\,M_{1,b}(z)+\dfrac{1}{b+1/2}\,M_{0,b}(z)\,W_{1,b}(z)
=
\dfrac{z}{\sqrt{\pi}}\,\dfrac{2^{2b}}{b+1/2}\,\Gamma(1+b),
\end{equation*}
which can be established through an astute use of various properties of the Whittaker functions. From this identity it is easy to see that~\eqref{eq:norm-factor-C-answer1} can be reduced down to
\begin{equation}\label{eq:norm-factor-C-answer2}
C
\triangleq
\sqrt{\pi}\;\dfrac{A}{4}\,e^{\tfrac{1}{A}}\,\dfrac{\xi(\lambda)+1}{2^{\xi(\lambda)}}\left/\Gamma\left(\dfrac{1}{2}\xi(\lambda)+1\right)\right.,
\end{equation}
whence, recalling again that $\xi(\lambda)\in[0,1)$, it can be concluded at once that $C>0$ for any $A>0$; the positivity of the normalizing factor $C$ is equivalent to saying that the integral of the general $q(x)$ given by~\eqref{eq:QSD-pdf-gen-answer0} with respect to $x$ over the interval $[A,+\infty)$ is always positive.

We are now in a position to make the following claim.
\begin{lemma}\label{lem:QSD-gen-answer}
For any fixed $A>0$ and arbitrary $\lambda\in(0,1/8]$, the function
\begin{equation}\label{eq:QSD-pdf-gen-answer}
\begin{split}
q(x)
&=
\dfrac{C}{x}\,e^{-\tfrac{1}{x}}\Biggl\{W_{1,\tfrac{1}{2}\xi(\lambda)}\left(\dfrac{2}{A}\right)M_{1,\tfrac{1}{2}\xi(\lambda)}\left(\dfrac{2}{x}\right)-\\
&\qquad\qquad\qquad\qquad
-M_{1,\tfrac{1}{2}\xi(\lambda)}\left(\dfrac{2}{A}\right)W_{1,\tfrac{1}{2}\xi(\lambda)}\left(\dfrac{2}{x}\right)\Biggr\}\,\indicator{x\in[A,+\infty)},
\end{split}
\end{equation}
with $\xi(\lambda)\in[0,1)$ as in~\eqref{eq:xi-def} and $C>0$ given by either~\eqref{eq:norm-factor-C-answer1} or~\eqref{eq:norm-factor-C-answer2}, solves the master equation~\eqref{eq:master-eqn}, and satisfies the absorbing boundary condition~\eqref{eq:QSD-pdf-bnd-cond-A} as well as the normalization constraint~\eqref{eq:QSD-pdf-norm-constraint}. Moreover, the following definite integral identity also holds:
\begin{equation}\label{eq:QSD-cdf-gen-answer}
\begin{split}
\int_{A}^{x}
q(y)\,dy
&=
1-C\,e^{-\tfrac{1}{x}}\Biggl\{W_{0,\tfrac{1}{2}\xi(\lambda)}\left(\dfrac{2}{x}\right)M_{1,\tfrac{1}{2}\xi(\lambda)}\left(\dfrac{2}{A}\right)+\\
&\qquad\qquad\qquad\qquad
+\dfrac{2}{\xi(\lambda)+1}\,M_{0,\tfrac{1}{2}\xi(\lambda)}\left(\dfrac{2}{x}\right)W_{1,\tfrac{1}{2}\xi(\lambda)}\left(\dfrac{2}{A}\right)\Biggr\},
\end{split}
\end{equation}
for any $x\in[A,+\infty)$.
\end{lemma}
\begin{proof}
We only need to show~\eqref{eq:QSD-cdf-gen-answer}. To that end, the result can be obtained by integrating~\eqref{eq:QSD-pdf-gen-answer0} with respect to $x$ and evaluating the integral with the aid of the definite integral identities~\eqref{eq:WhitM-int-formula} and~\eqref{eq:WhitW-int-formula}.
\end{proof}

The only question that has not yet been addressed is that of actually finding $\lambda$. To that end, recall that $\varphi(x,\lambda)$ given by~\eqref{eq:eigfun-gen-form} must be square-integrable with respect to the $\mathfrak{m}(x)$ measure given by~\eqref{eq:speed-measure}. More concretely, this means that $\varphi(x,\lambda)$ must be such that $\|\varphi(\cdot,\lambda)\|<+\infty$ where the norm $\|\varphi(\cdot,\lambda)\|$ is defined in~\eqref{eq:eigfun-norm-def}. Due to the Whittaker $W$ and $M$ functions' near-origin behavior given by~\eqref{eq:WhitW-near-zero-asymptotics} and~\eqref{eq:WhitM-near-zero-asymptotics}, respectively, the function $\varphi(x,\lambda)$ is not square-integrable near $x=+\infty$, unless either $\lambda=1/8$ or $B_2=0$ in the right-hand side of~\eqref{eq:eigfun-gen-form}. As a result, we are to distinguish two separate cases: (a) $\lambda\in(0,1/8)$ and $B_2=0$ in~\eqref{eq:eigfun-gen-form}, and (b) $\lambda=1/8$.

Let us first suppose that $\lambda\in(0,1/8)$. Since in this case we must set $B_2=0$ in~\eqref{eq:eigfun-gen-form} to achieve $\|\varphi(\cdot,\lambda)\|<+\infty$, it follows that we must also set $C_2=0$ in~\eqref{eq:QSD-pdf-gen-form}. If $C_2=0$ then from~\eqref{eq:C1C2-eqn1} it is clear that the only way for $C_1$ to have a nontrivial value is to demand that $\lambda$ and $A>0$ be connected via the equation
\begin{equation}\label{eq:lambdaA-eqn}
M_{1,\tfrac{1}{2}\xi(\lambda)}\left(\dfrac{2}{A}\right)
=
0,
\end{equation}
where $\xi(\lambda)$ is as in~\eqref{eq:xi-def}. Although this equation does not permit a closed form solution $\lambda\equiv\lambda_A$ as a function of $A>0$, it can be gleaned from~\cite[p.~182]{Buchholz:Book1969} that this equation does have at most one solution $\lambda\equiv\lambda_A\in(0,1/8]$ for any $A>0$. More concretely, this solution is an increasing function of $A>0$, and it ceases to exist as soon as $A>0$ reaches the value $A^{*}$ that is the solution of the equation
\begin{equation}\label{eq:Astar-eqn}
M_{1,0}\left(\dfrac{2}{A^{*}}\right)
=
0,
\end{equation}
and although this equation is also transcendental, it is easily solvable numerically with any desired accuracy, yielding $A^{*}\approx1.265857361$.

We wrote a {\em Mathematica} script to solve equation~\eqref{eq:lambdaA-eqn} for $\lambda_A$ numerically. Figure~\ref{fig:lambda_vs_A} was obtained with the help of our {\em Mathematica} script, and it shows the behavior of $\lambda_A$ as a function of $A\in(0,A^{*}]$. It is clear from the figure that $\lambda_A\;(>0)$ is an increasing function of $A\;(>0)$ rapidly growing up to the value of $1/8=0.125$, which is the cutoff point of the spectrum of the operator $\mathscr{D}$ given by~\eqref{eq:Doperator-def}. Moreover, the figure also shows that the value of $1/8$ is attained by $\lambda_A$ at $A=A^{*}\approx1.265857361$ where $A^{*}$ is the solution of equation~\eqref{eq:Astar-eqn}. All this not only fulfills but also improves some of the predictions previously made in~\cite[Section~7.8.2]{Collet+etal:Book2013}, where, in particular, it was shown that $\lambda_A=1/8$ is guaranteed for $A\ge e^3\approx 20.0855369$.
\begin{figure}[h!]
    \centering
    \includegraphics[width=\linewidth]{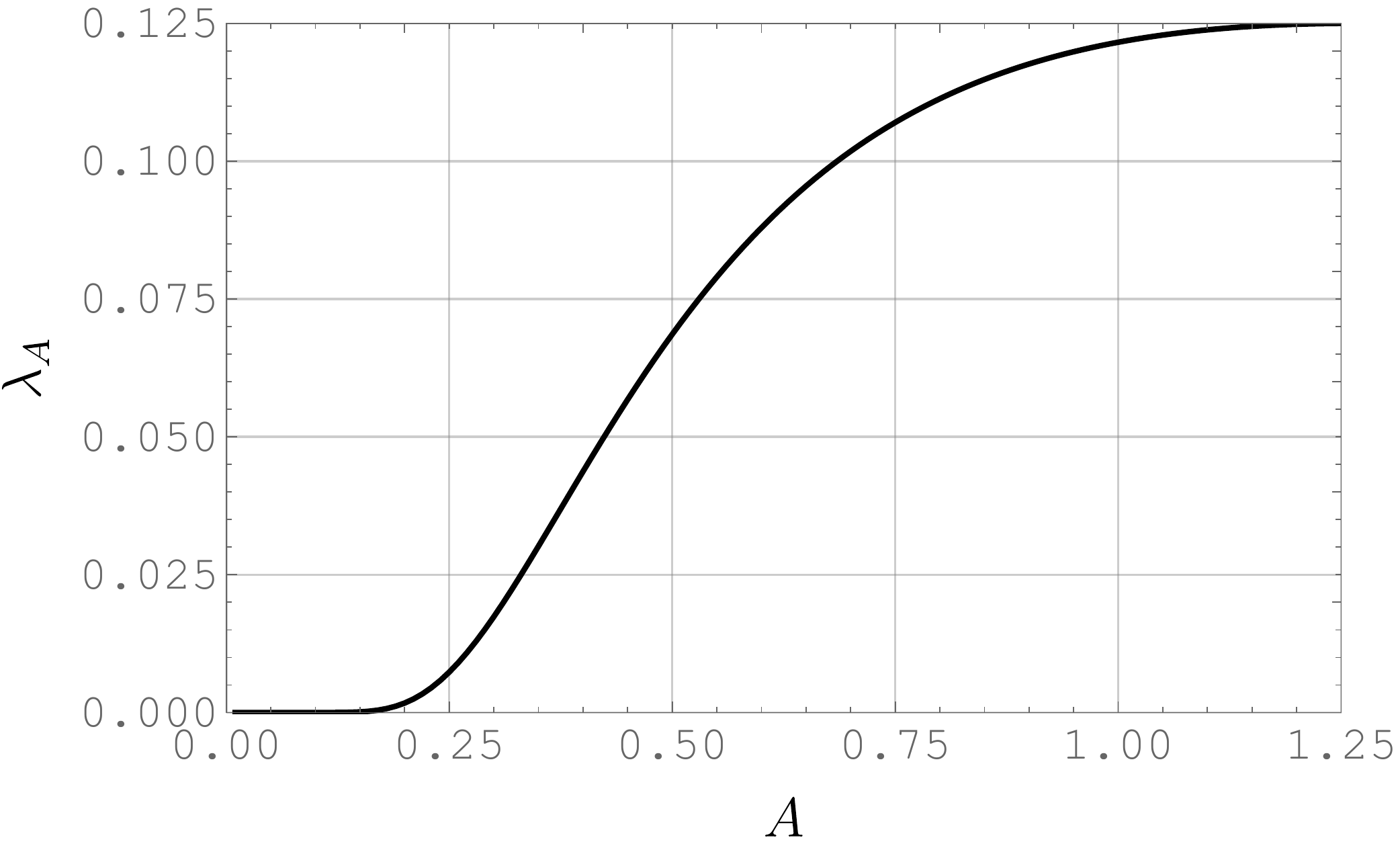}
    \caption{Smallest eigenvalue $\lambda_A\in(0,1/8]$ of Sturm--Liouville operator $\mathscr{D}$ as a function of $A\in(0,A^{*}]$.}
    \label{fig:lambda_vs_A}
\end{figure}

We can now conclude that, if $A\in(0,A^{*})$, then the quasi-stationary distribution's pdf and cdf are given by~\eqref{eq:QSD-pdf-gen-answer} and~\eqref{eq:QSD-cdf-gen-answer}, respectively, with $\lambda_A\in(0,1/8)$ determined as the only solution of equation~\eqref{eq:lambdaA-eqn}.

Let us now switch attention to the case when $\lambda=1/8$. From the above discussion it follows that this case takes effect for $A\ge A^{*}\approx1.265857361$. The quasi-stationary distribution formulae~\eqref{eq:QSD-pdf-gen-answer} and~\eqref{eq:QSD-cdf-gen-answer} remain valid ``as is'', except that $\xi(\lambda)$ becomes $\xi(1/8)=0$.

At this point we have effectively proved the following two theorems.
\begin{theorem}\label{thm:QSD-answer1}
If $A^{*}\approx1.265857361$ is the solution of the equation
\begin{equation*}
M_{1,0}\left(\dfrac{2}{A^{*}}\right)
=
0,
\end{equation*}
then for every fixed $A\in(0,A^{*})$ the equation
\begin{equation*}
M_{1,\tfrac{1}{2}\xi(\lambda)}\left(\dfrac{2}{A}\right)
=
0
\;\;
\text{with}
\;\;
\xi(\lambda)\triangleq\sqrt{1-8\lambda}
\end{equation*}
has exactly one solution $\lambda\equiv\lambda_A\in(0,1/8)$, and the quasi-stationary density $q_A(x)$ is given by
\begin{equation}\label{eq:QSD-pdf-answer1}
q_A(x)
=
\dfrac{\dfrac{\xi(\lambda)+1}{x}\,e^{-\tfrac{1}{x}}\,M_{1,\tfrac{1}{2}\xi(\lambda)}\left(\dfrac{2}{x}\right)}{2\,e^{-\tfrac{1}{A}}\,M_{0,\tfrac{1}{2}\xi(\lambda)}\left(\dfrac{2}{ A}\right)}\,\indicator{x\in[A,+\infty)},
\end{equation}
while the respective quasi-stationary cdf $Q_A(x)$ is given by
\begin{equation}\label{eq:QSD-cdf-answer1}
Q_A(x)
=
\begin{cases}
&1-\dfrac{e^{-\tfrac{1}{x}}\,M_{0,\tfrac{1}{2}\xi(\lambda)}\left(\dfrac{2}{x}\right)}{e^{-\tfrac{1}{A}}\,M_{0,\tfrac{1}{2}\xi(\lambda)}\left(\dfrac{2}{A}\right)},\;\text{for $x\in(A,+\infty)$;}\\[2mm]
&0,\;\text{otherwise}.\\[2mm]
\end{cases}
\end{equation}
%Moreover, the pdf $q_A(x)$ and the cdf $Q_A(x)$ also permit the following alternative representations:
\end{theorem}
\begin{theorem}\label{thm:QSD-answer2}
If $A^{*}\approx1.265857361$ is the solution of the equation
\begin{equation*}
M_{1,0}\left(\dfrac{2}{A^{*}}\right)
=
0,
\end{equation*}
then for every fixed $A\ge A^{*}$ the equation
\begin{equation*}
M_{1,\tfrac{1}{2}\xi(\lambda)}\left(\dfrac{2}{A}\right)
=
0
\;\;
\text{with}
\;\;
\xi(\lambda)\triangleq\sqrt{1-8\lambda}
\end{equation*}
has no solution $\lambda\equiv\lambda_A\in(0,1/8]$ except $\lambda\equiv\lambda_A=1/8$ attained at $A=A^{*}$, and the quasi-stationary density $q_A(x)$ is given by
\begin{equation}\label{eq:QSD-pdf-answer2}
\begin{split}
q_A(x)
&=
\sqrt{\pi}\,e^{\tfrac{1}{A}}\,\dfrac{A}{4x}\,e^{-\tfrac{1}{x}}\Biggl\{W_{1,0}\left(\dfrac{2}{A}\right)M_{1,0}\left(\dfrac{2}{x}\right)-\\
&\qquad\qquad\qquad\qquad\qquad\qquad
-M_{1,0}\left(\dfrac{2}{A}\right)W_{1,0}\left(\dfrac{2}{x}\right)\Biggr\}\,\indicator{x\in[A,+\infty)},
\end{split}
\end{equation}
while the respective quasi-stationary cdf $Q_A(x)$ is given by
\begin{equation}\label{eq:QSD-cdf-answer2}
Q_A(x)
=
\begin{cases}
&1-\sqrt{\pi}\,e^{\tfrac{1}{A}}\,\dfrac{A}{4}\,e^{-\tfrac{1}{x}}\Biggl\{W_{0,0}\left(\dfrac{2}{x}\right)M_{1,0}\left(\dfrac{2}{A}\right)+\\[2mm]
&\qquad\qquad\qquad\qquad
+2\,M_{0,0}\left(\dfrac{2}{x}\right)W_{1,0}\left(\dfrac{2}{A}\right)\Biggr\},\;\text{for $x\in(A,+\infty)$;}\\[2mm]
&0,\;\text{otherwise}.\\[2mm]
\end{cases}
\end{equation}
\end{theorem}

Theorems~\ref{thm:QSD-answer1} and~\ref{thm:QSD-answer2} both easily follow directly from Lemma~\ref{lem:QSD-gen-answer} and the discussion preceding it. It is also worth pointing out that, in view of~\eqref{eq:norm-factor-C-answer1} and~\eqref{eq:norm-factor-C-answer2}, formulae~\eqref{eq:QSD-pdf-answer1} and~\eqref{eq:QSD-cdf-answer1} each permit a different expression, similar to that of formulae~\eqref{eq:QSD-pdf-answer2} and~\eqref{eq:QSD-cdf-answer2}, respectively. The possibility of formulae~\eqref{eq:QSD-pdf-answer1} and~\eqref{eq:QSD-cdf-answer1} taking the form akin to that of formulae~\eqref{eq:QSD-pdf-answer2} and~\eqref{eq:QSD-cdf-answer2}, respectively, is an indication that the quasi-stationary distribution is a smooth, continuous function of $A>0$.

As complicated as the obtained formulae~\eqref{eq:QSD-pdf-answer1}--\eqref{eq:QSD-cdf-answer1} and~\eqref{eq:QSD-pdf-answer2}--\eqref{eq:QSD-cdf-answer2} may seem, they are all perfectly amenable to numerical evaluation, meaning that the quasi-stationary distribution's pdf and cdf can all be evaluated numerically to within any desired accuracy, for any $A>0$. To illustrate this point, we implemented the formulae in a {\em Mathematica} script, and used it to perform a few numerical experiments each corresponding to a specific value of $A>0$. The obtained results are presented next.

Figures~\ref{fig:qA+QA_A_1_over_10},~\ref{fig:qA+QA_A_1_over_2}, and~\ref{fig:qA+QA_A_1} depict the quasi-stationary pdf $q_A(x)$ and the corresponding cdf $Q_A(x)$ as functions of $1/x\in[0,1/A]$ for $A=1/10$, $A=1/2$, and for $A=1$, respectively. We stress that the $x$-axis scale is not $x\in[A,+\infty)$ but is $1/x\in[0,1/A]$. This is intentional, and is done to achieve finiteness of the domain of the quasi-stationary distribution. On the flip side, however, this transformation effectively reverses the direction of the $x$-axis, which is why the cdf $Q_A(x)$ appears as a decreasing function of $x$: it is not, as long as one keeps in mind that the $x$-axis is the reciprocal of $x$. It is also of note that $A=\{1/10,1/2,1\}$ are all smaller than $A^{*}\approx1.265857361$, so that $\lambda_A\in(0,1/8)$ and the corresponding pdf $q(x)$ and cdf $Q(x)$ are given by Theorem~\ref{thm:QSD-answer1}.
\begin{figure}[tph]
    \centering
    \subfloat[$q_A(x)$.]{%\label{fig:qA_A_1_over_10}
        \includegraphics[width=0.47\textwidth]{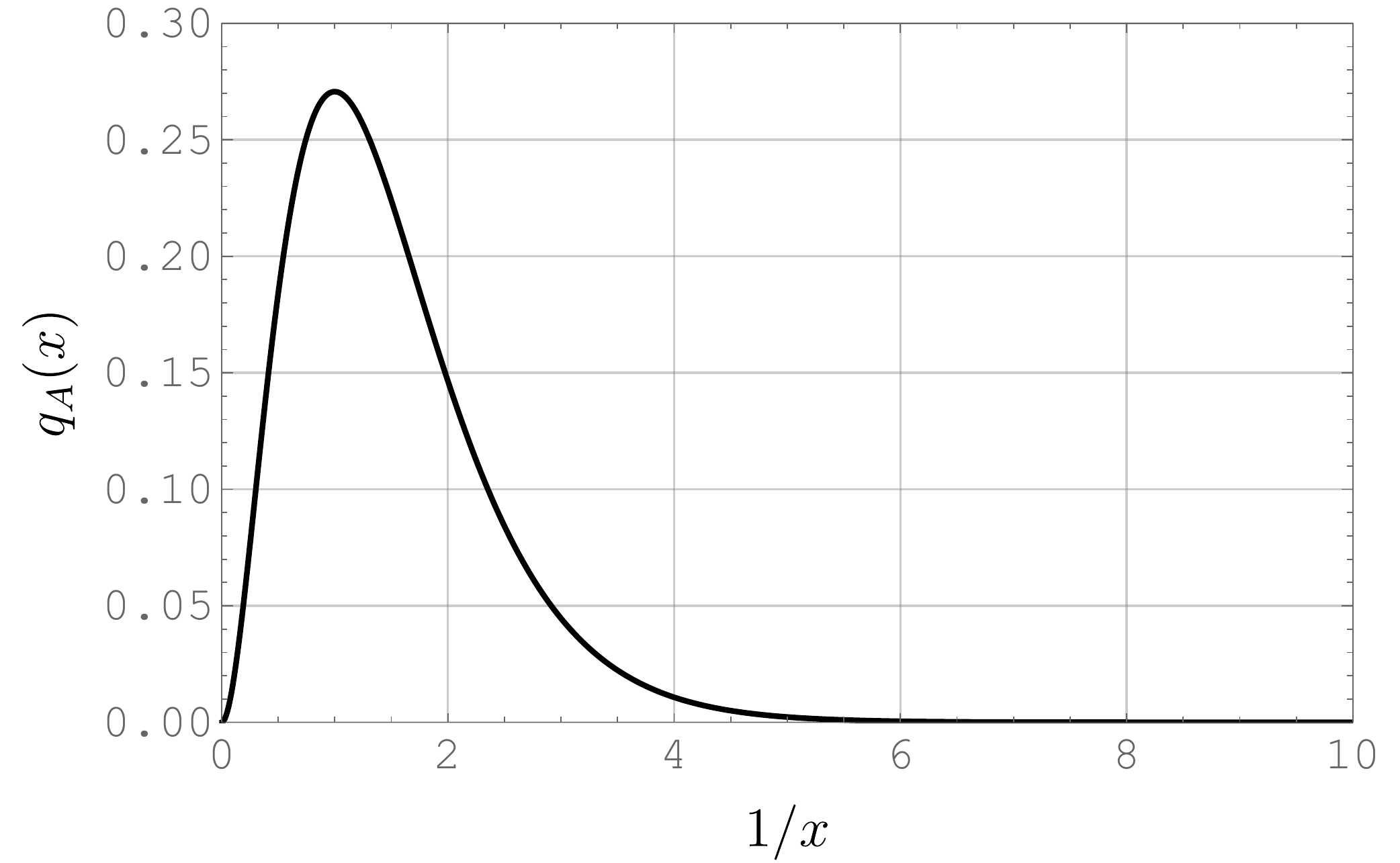}
    }
    \;
    \subfloat[$Q_A(x)$.]{%\label{fig:QA_A_1_over_10}
        \includegraphics[width=0.47\textwidth]{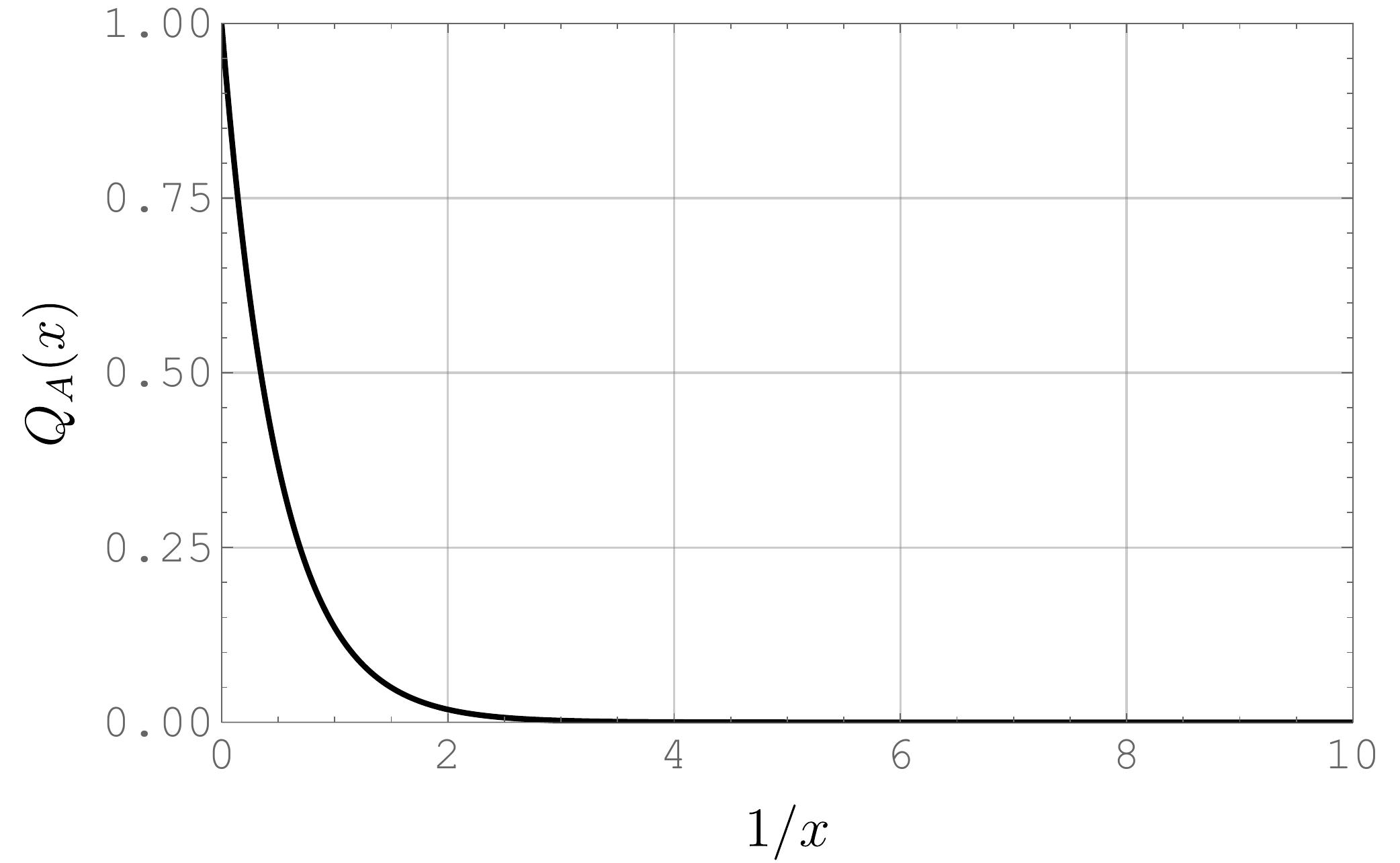}
    }
    \caption{Quasi-stationary distribution's pdf $q_A(x)$ and cdf $Q_A(x)$ as functions of $1/x$ for $A=1/10$.}
    \label{fig:qA+QA_A_1_over_10}
\end{figure}
\begin{figure}[tph]
    \centering
    \subfloat[$q_A(x)$.]{%\label{fig:qA_A_1_over_2}
        \includegraphics[width=0.47\textwidth]{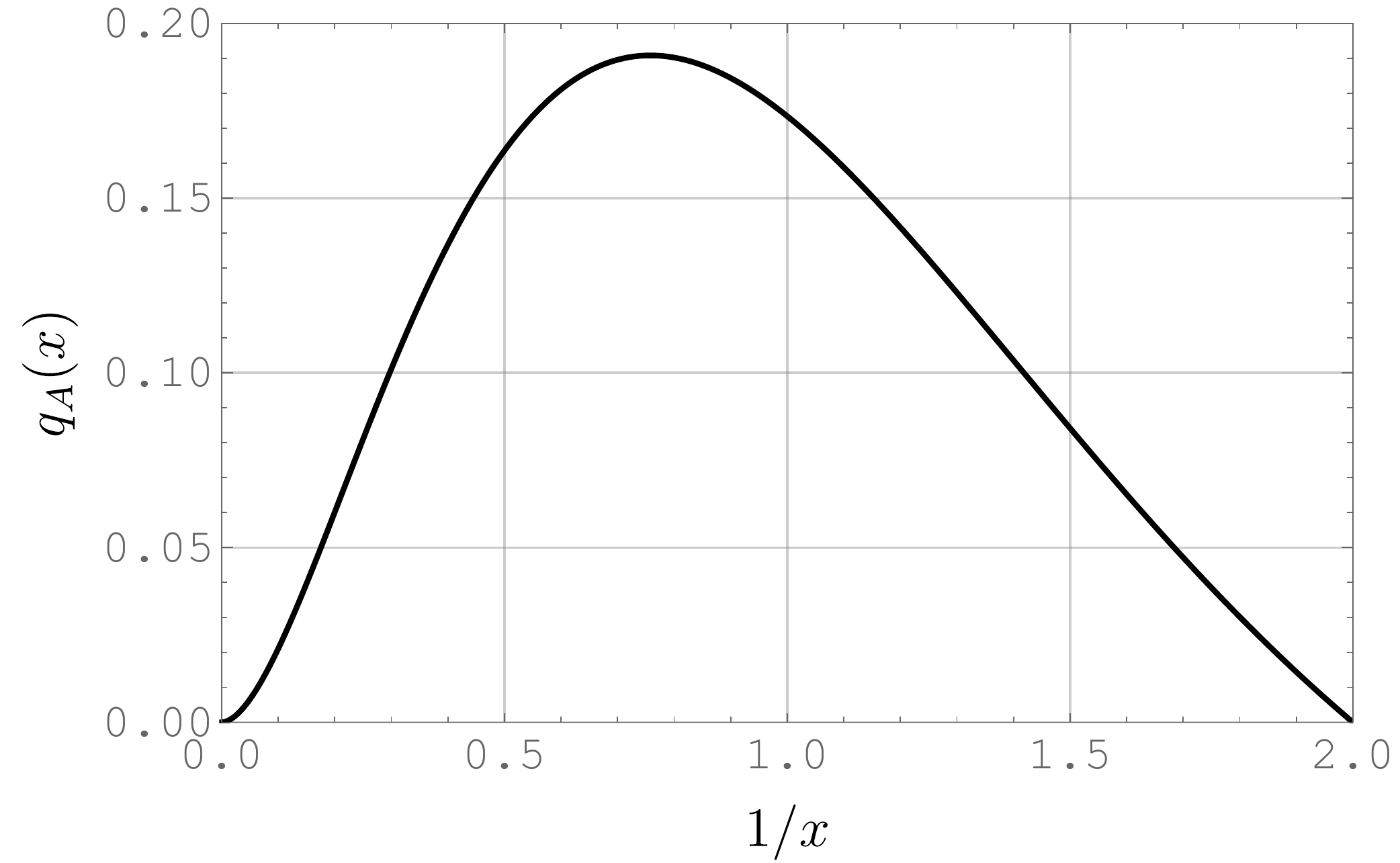}
    }
    \;
    \subfloat[$Q_A(x)$.]{%\label{fig:QA_A_1_over_2}
        \includegraphics[width=0.47\textwidth]{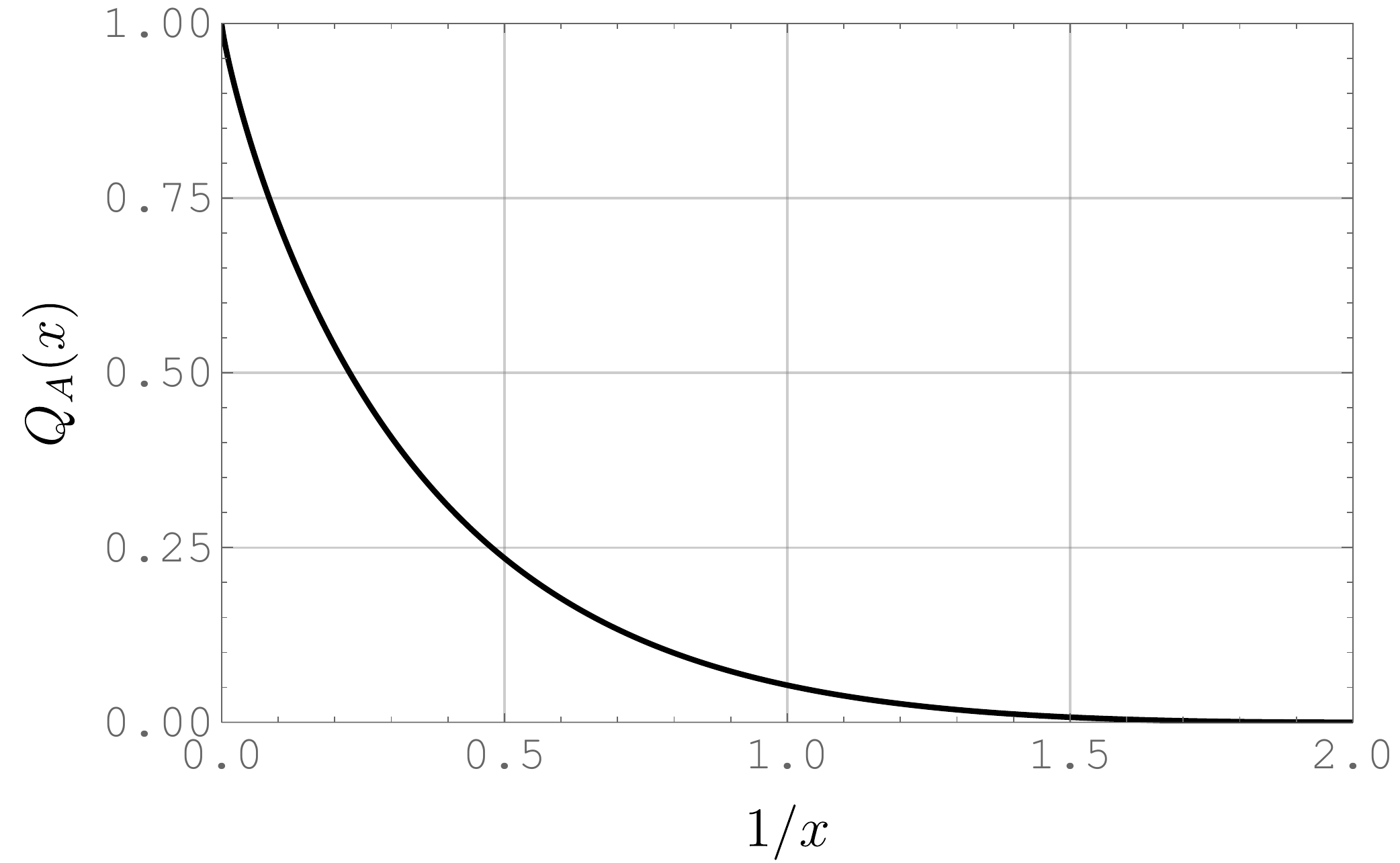}
    }
    \caption{Quasi-stationary distribution's pdf $q_A(x)$ and cdf $Q_A(x)$ as functions of $1/x$ for $A=1/2$.}
    \label{fig:qA+QA_A_1_over_2}
\end{figure}
\begin{figure}[tph]
    \centering
    \subfloat[$q_A(x)$.]{%\label{fig:qA_A_1}
        \includegraphics[width=0.47\textwidth]{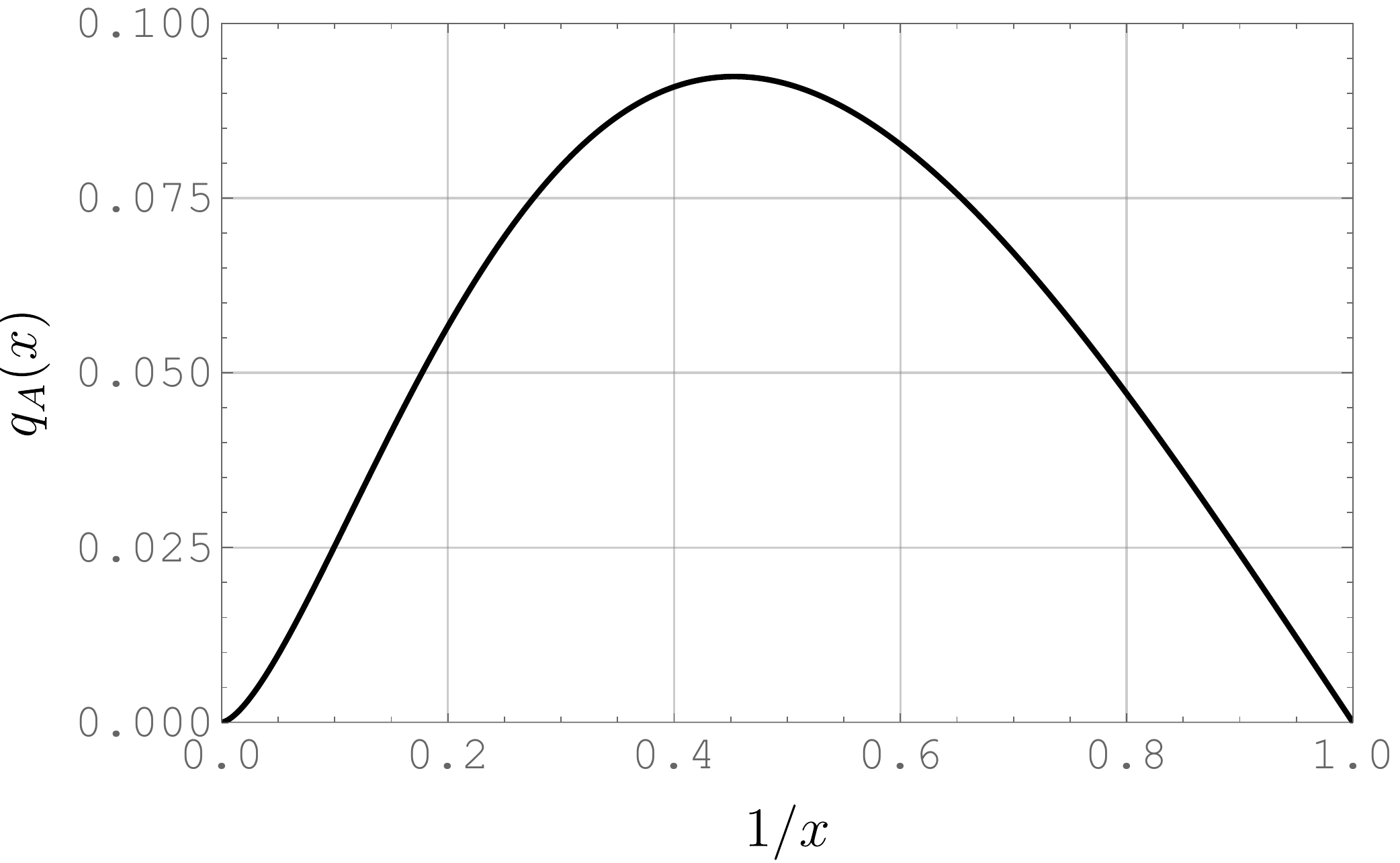}
    }
    \;
    \subfloat[$Q_A(x)$.]{%\label{fig:QA_A_1}
        \includegraphics[width=0.47\textwidth]{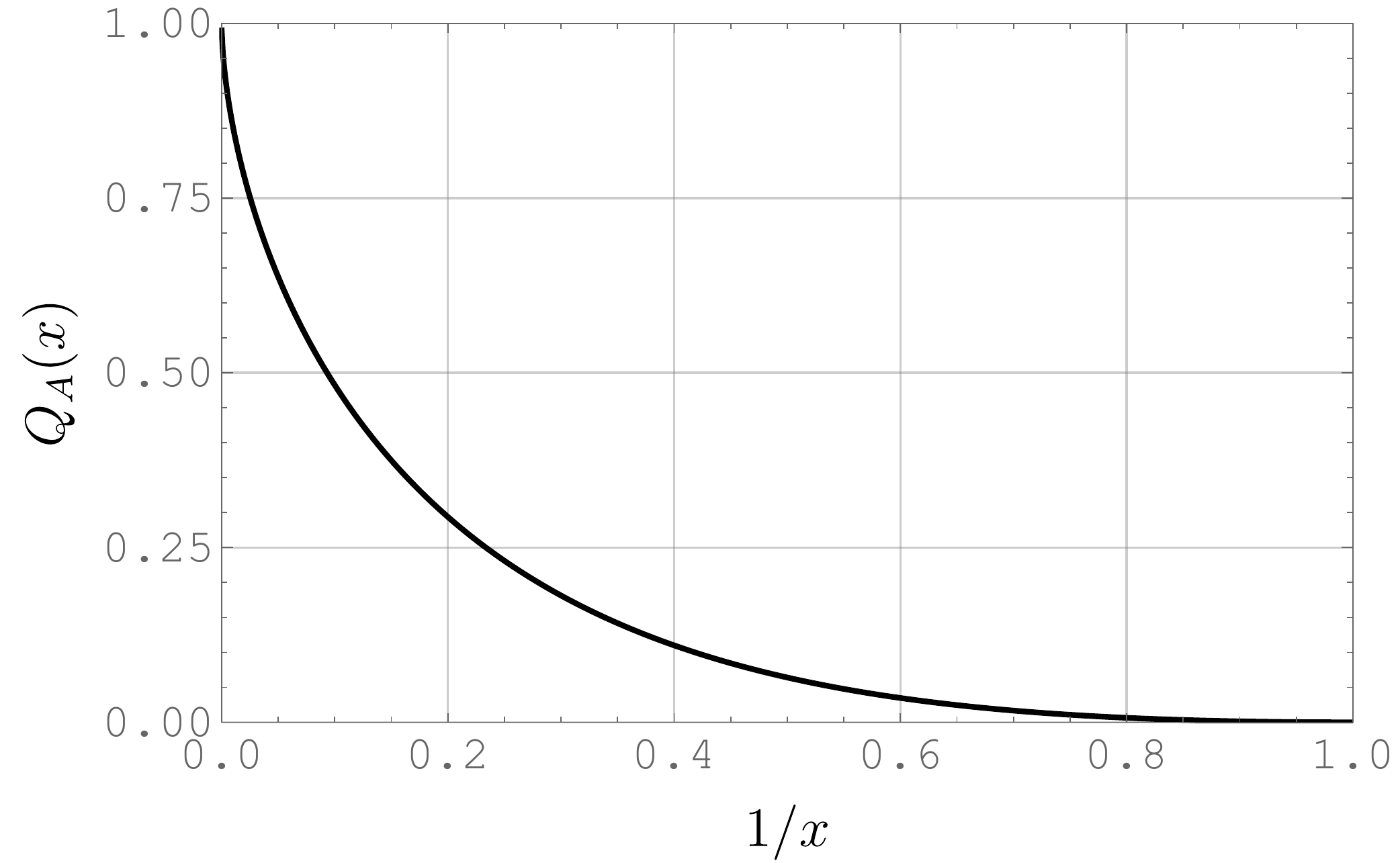}
    }
    \caption{Quasi-stationary distribution's pdf $q_A(x)$ and cdf $Q_A(x)$ as functions of $1/x$ for $A=1$.}
    \label{fig:qA+QA_A_1}
\end{figure}

Figures~\ref{fig:qA+QA_A_2},~\ref{fig:qA+QA_A_5}, and~\ref{fig:qA+QA_A_10} show the quasi-stationary pdf $q(x)$ and cdf $Q(x)$ for $A=2$, $A=5$, and $A=10$, respectively. Since $A=\{2,5,10\}$ are all higher than $A^{*}\approx1.265857361$, it follows that $\lambda=1/8$ and the corresponding pdf $q(x)$ and cdf $Q(x)$ are given by Theorem~\ref{thm:QSD-answer2}.
\begin{figure}[tph]
    \centering
    \subfloat[$q_A(x)$.]{%\label{fig:qA_A_2}
        \includegraphics[width=0.47\textwidth]{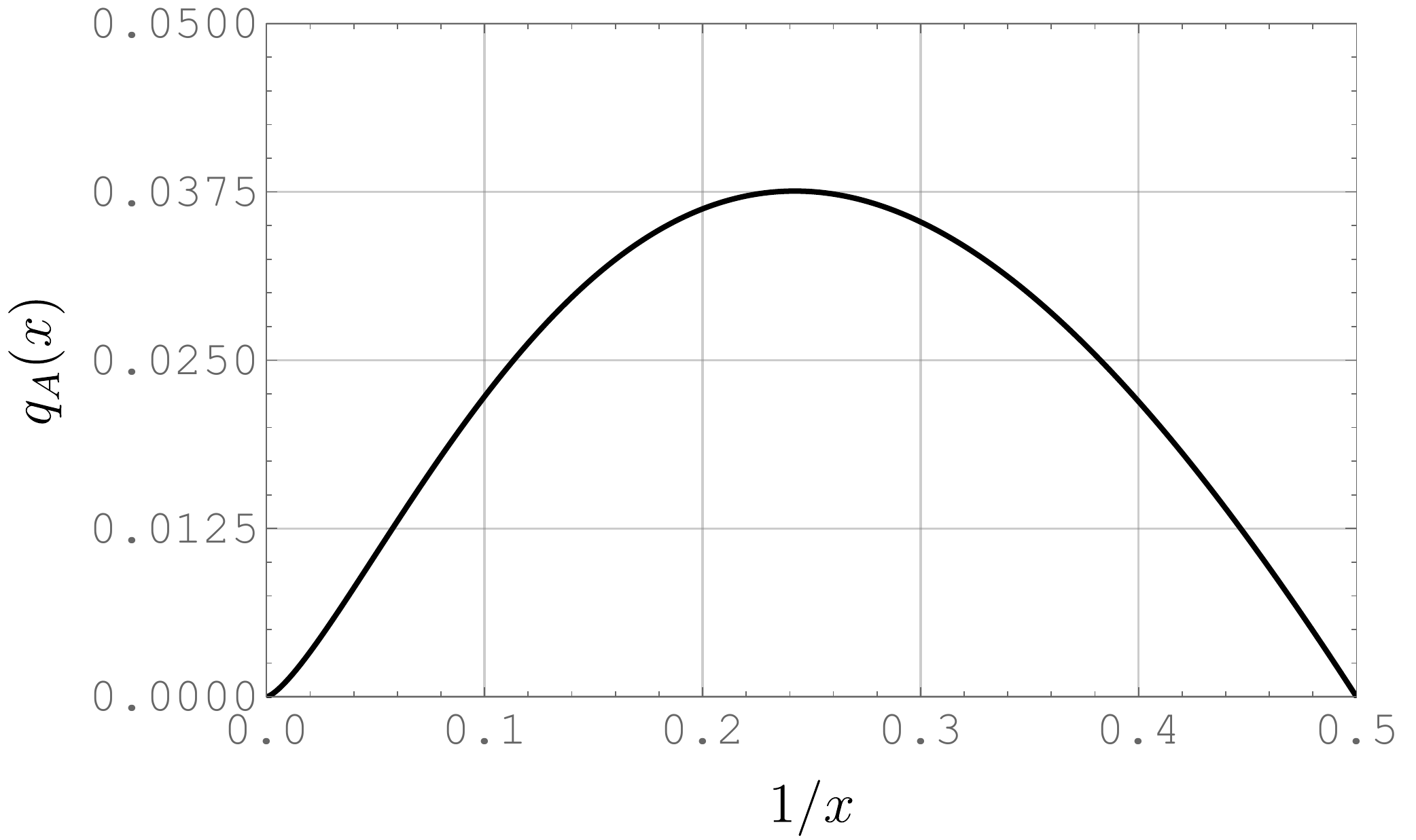}
    }
    \;
    \subfloat[$Q_A(x)$.]{%\label{fig:QA_A_2}
        \centering
        \includegraphics[width=0.47\textwidth]{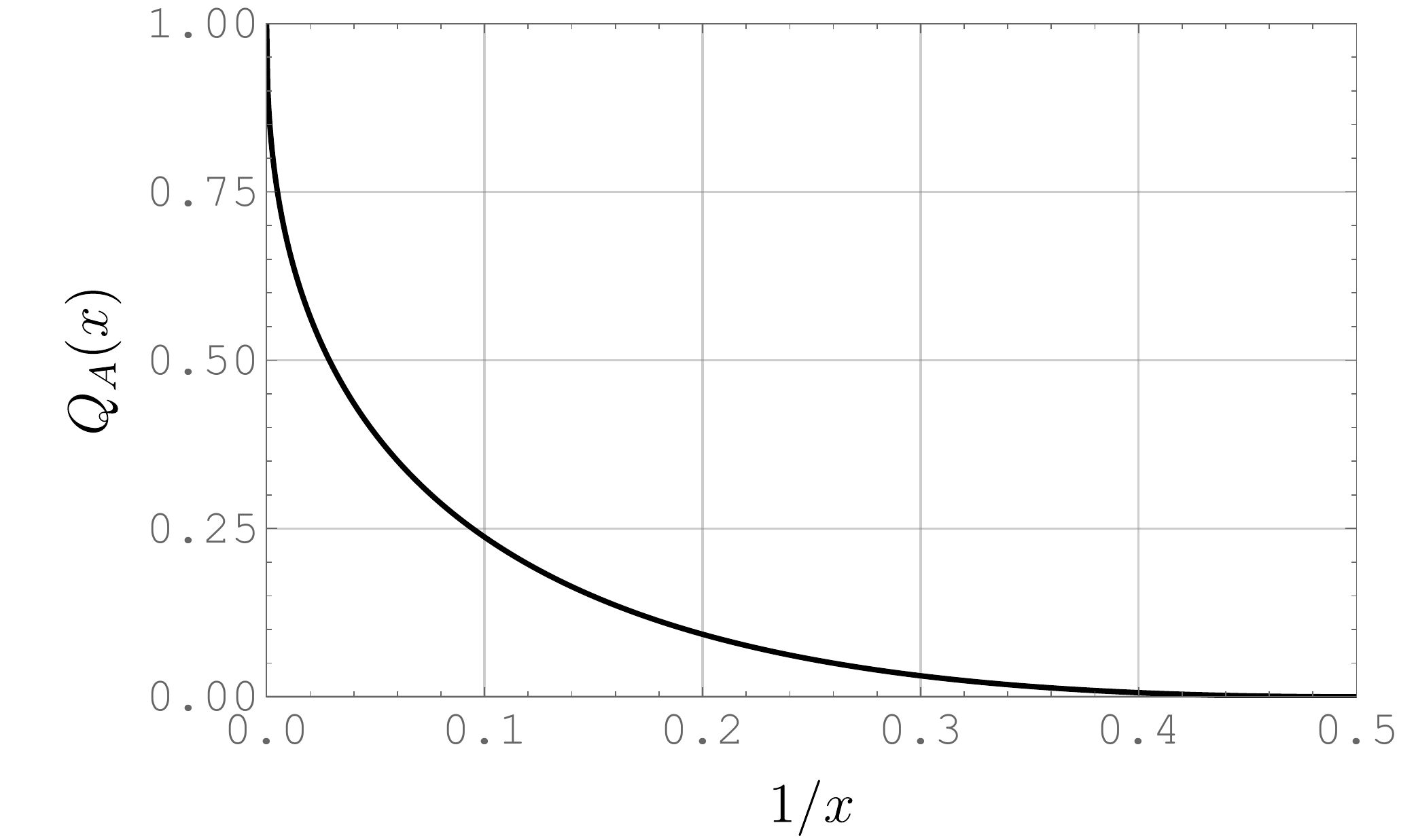}
    }
    \caption{Quasi-stationary distribution's pdf $q_A(x)$ and cdf $Q_A(x)$ as functions of $1/x$ for $A=2$.}
    \label{fig:qA+QA_A_2}
\end{figure}
\begin{figure}[tph]
    \centering
    \subfloat[$q_A(x)$.]{%\label{fig:qA_A_5}
        \includegraphics[width=0.47\textwidth]{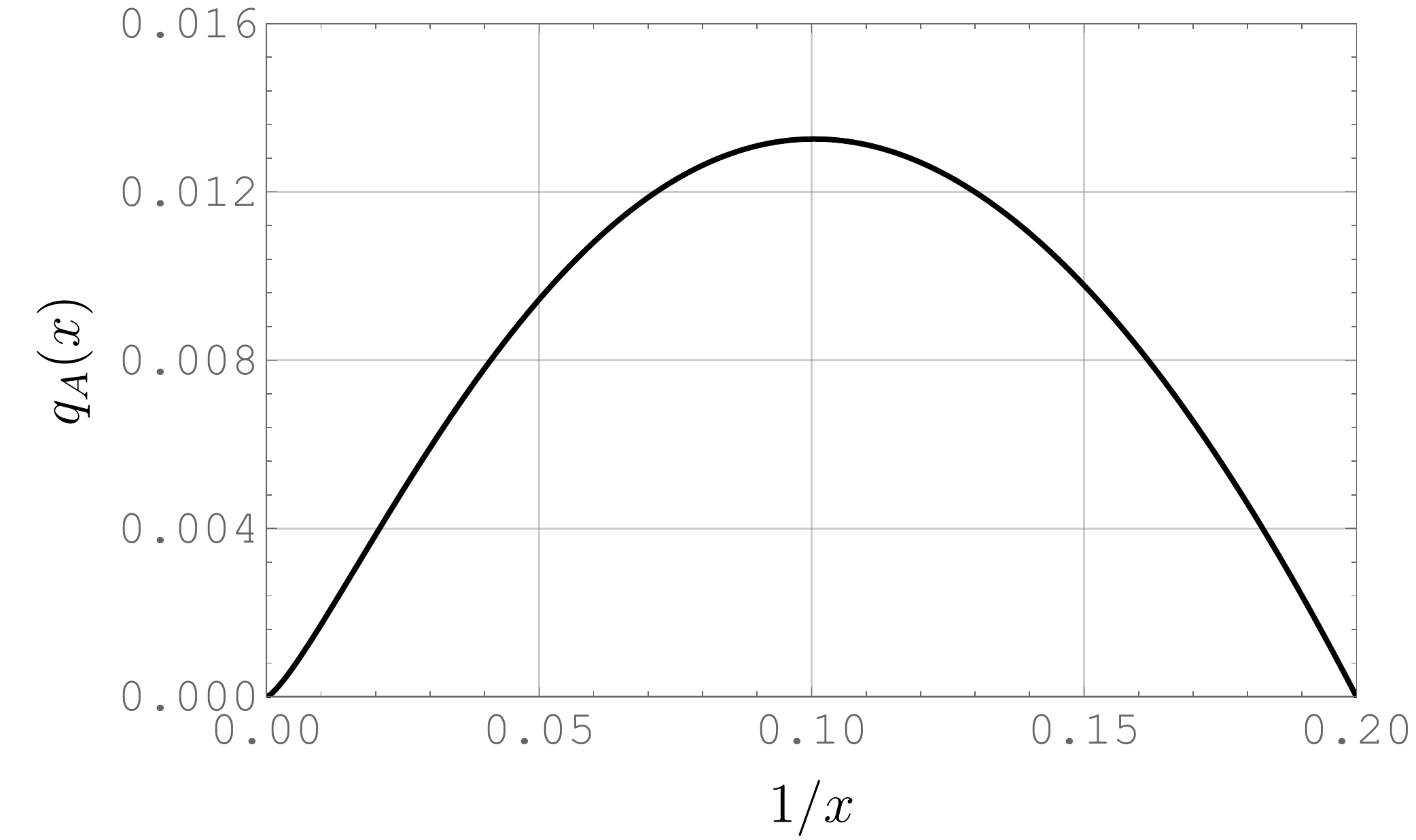}
    }
    \;
    \subfloat[$Q_A(x)$.]{%\label{fig:QA_A_5}
        \includegraphics[width=0.47\textwidth]{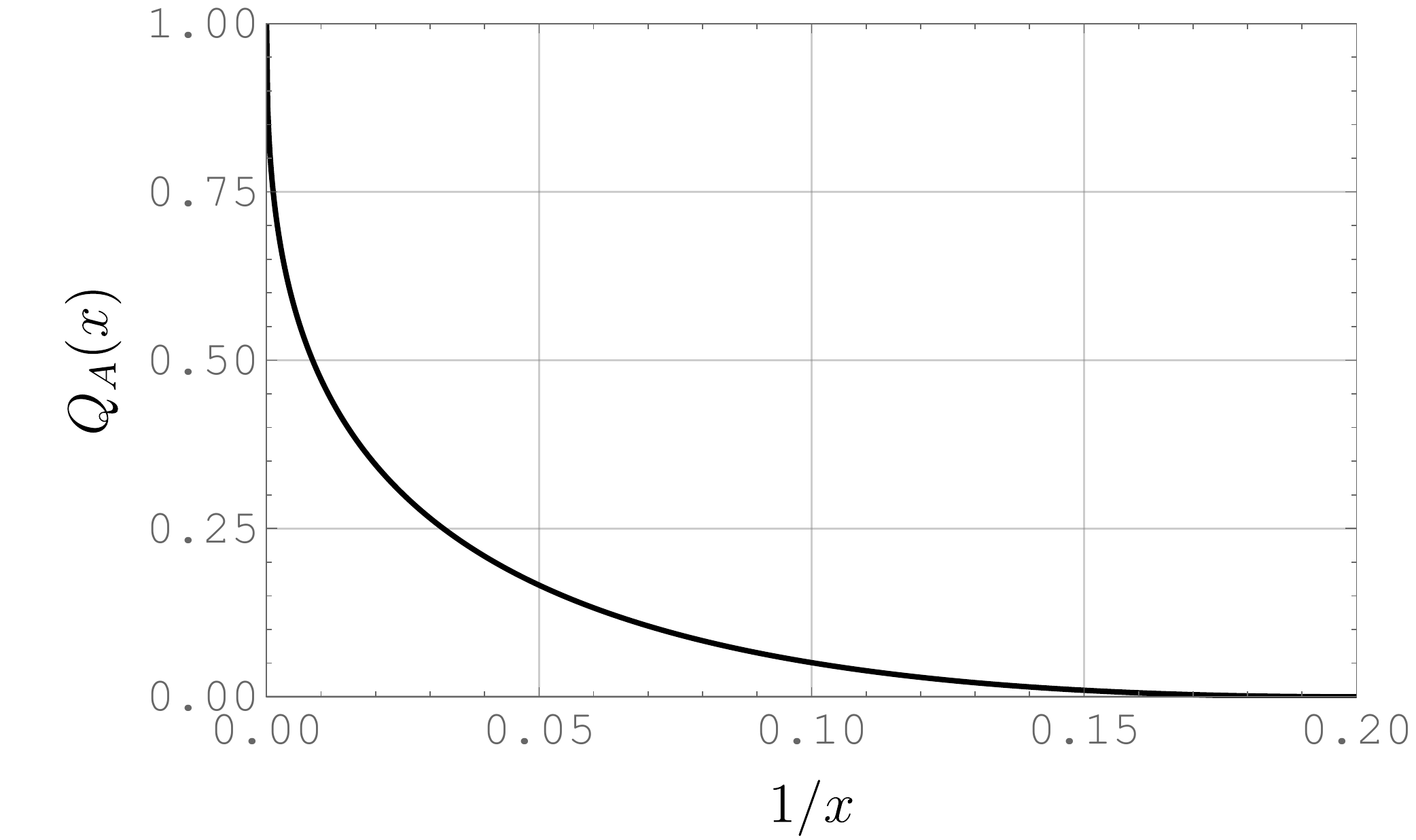}
    }
    \caption{Quasi-stationary distribution's pdf $q_A(x)$ and cdf $Q_A(x)$ as functions of $1/x$ for $A=5$.}
    \label{fig:qA+QA_A_5}
\end{figure}
\begin{figure}[tph]
    \centering
    \subfloat[$q_A(x)$.]{%\label{fig:qA_A_10}
        \includegraphics[width=0.47\textwidth]{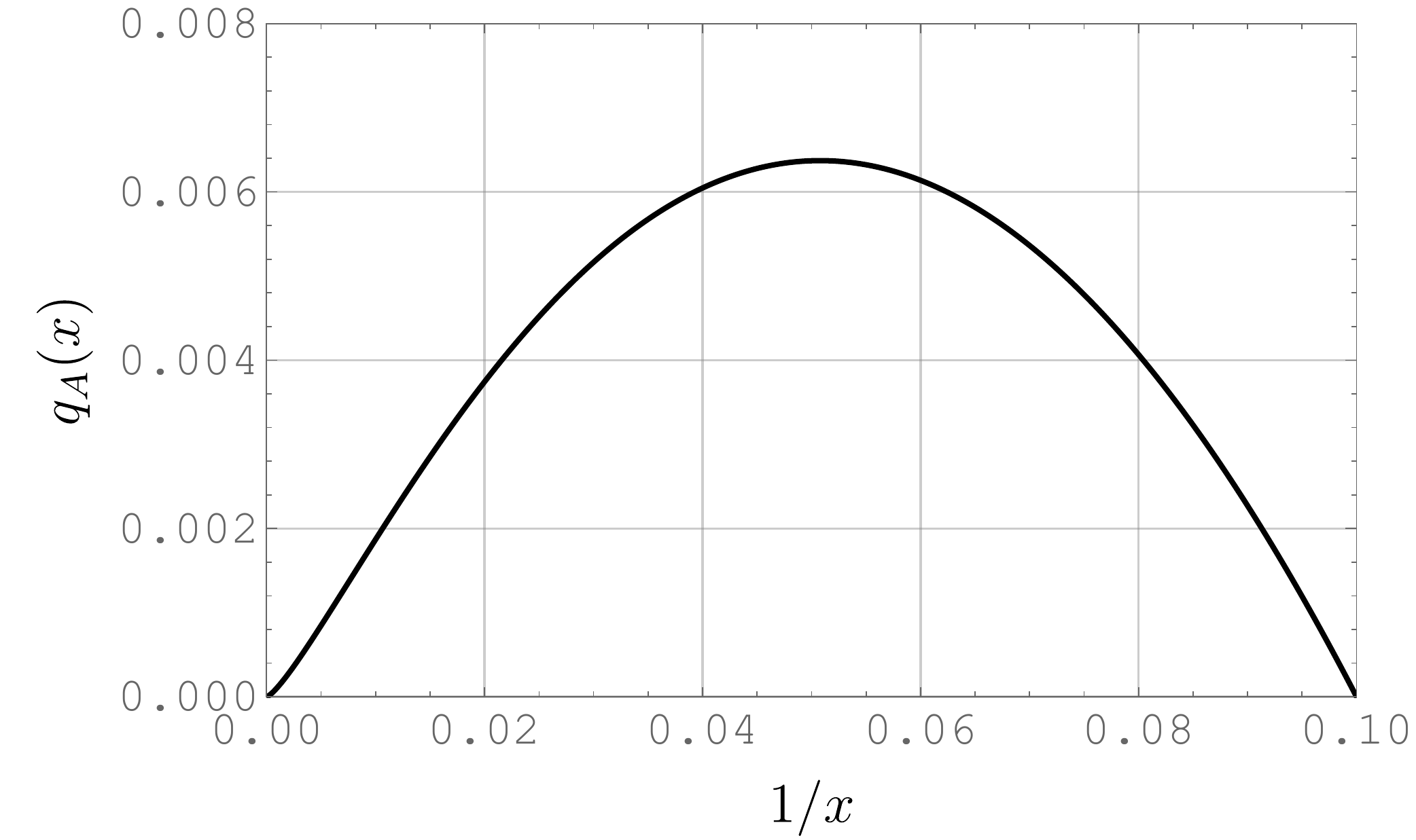}
    }
    \;
    \subfloat[$Q_A(x)$.]{%\label{fig:QA_A_10}
        \includegraphics[width=0.47\textwidth]{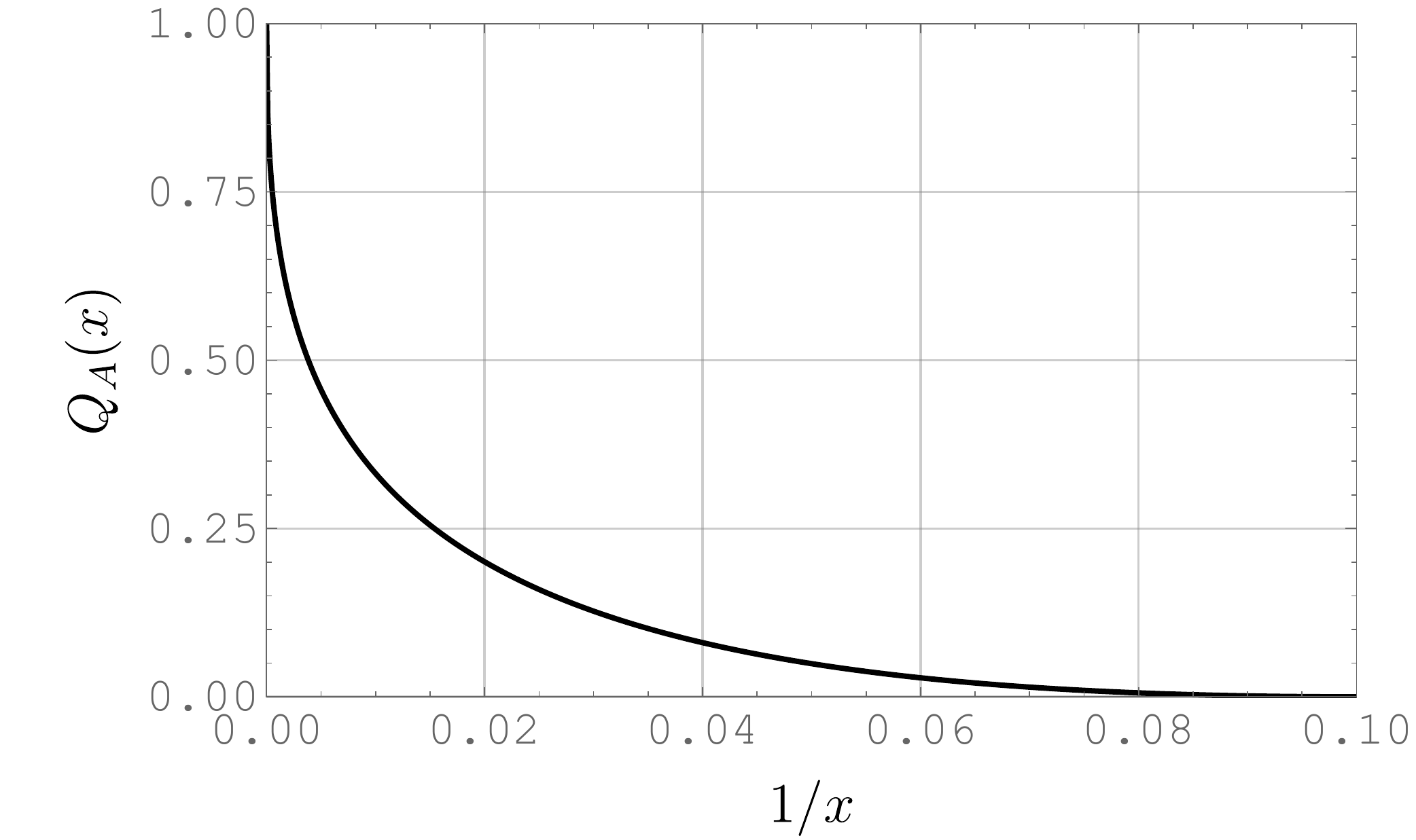}
    }
    \caption{Quasi-stationary distribution's pdf $q_A(x)$ and cdf $Q_A(x)$ as functions of $1/x$ for $A=10$.}
    \label{fig:qA+QA_A_10}
\end{figure}

To draw a line under the entire paper, we remark that formulae~\eqref{eq:QSD-pdf-gen-answer0}--\eqref{eq:norm-factor-C-answer2} actually give a whole family of quasi-stationary densities $q_A(x)$ indexed by $\lambda\in(0,\lambda_A)$ where $\lambda_A\in(0,1/8]$ is the ``bottom'' of the spectrum of the operator $\mathscr{D}$ defined in~\eqref{eq:Doperator-def}. Put another way, for any fixed $A>0$ and any fixed $\lambda\in(0,\lambda_A)$, the function $q_A(x)$ given by formulae~\eqref{eq:QSD-pdf-gen-answer0}--\eqref{eq:norm-factor-C-answer2} is ``legitimate'' pdf supported on $[A,+\infty)$, because it is nonnegative for any $x\in[A,+\infty)$ and integrates to unity over the interval $[A,+\infty)$. Indeed, first note that if $\lambda\in(0,\lambda_A)$, then $q_A(x)$ must be different from zero for all $x>A$, for otherwise $\lambda_A$ would not be the smallest eigenvalue. Therefore $q_A(x)$ is either positive or negative for all $x>A$. To see that $q_A(x)$ cannot be negative, observe that from~\eqref{eq:QSD-cdf-gen-answer} we have
\begin{equation*}
\begin{split}
\int_{A}^{+\infty}
q(y)\,dy
&=
1-C\lim_{x\to0+} \Biggl\{e^{-x}\Biggl[W_{0,\tfrac{1}{2}\xi(\lambda)}\left(2x\right)M_{1,\tfrac{1}{2}\xi(\lambda)}\left(\dfrac{2}{A}\right)+\\
&\qquad\qquad\qquad\qquad
+\dfrac{2}{\xi(\lambda)+1}\,M_{0,\tfrac{1}{2}\xi(\lambda)}\left(2x\right)W_{1,\tfrac{1}{2}\xi(\lambda)}\left(\dfrac{2}{A}\right)\Biggr]\Biggr\},
\end{split}
\end{equation*}
where we recall that $C>0$ is a constant (independent of $x$ but dependent on $A$) given by either~\eqref{eq:norm-factor-C-answer1} or~\eqref{eq:norm-factor-C-answer2}. However, according to~\eqref{eq:WhitW-near-zero-asymptotics} and~\eqref{eq:WhitM-near-zero-asymptotics}, the limit in the right-hand side of the foregoing formula is zero, because $\xi(\lambda)\in[0,1)$ whenever $\lambda\in(0,1/8]$, as can be easily seen from~\eqref{eq:xi-def}. Thus $q_A(x)$ given by~\eqref{eq:QSD-pdf-gen-answer0}--\eqref{eq:norm-factor-C-answer2} integrates to one over the interval $[A,+\infty)$. This necessitates that the sign maintained by $q_A(x)$ in the interior of this interval be positive. We therefore arrive at the curious conclusion: formulae~\eqref{eq:QSD-pdf-gen-answer0}--\eqref{eq:norm-factor-C-answer2} yield a ``legitimate'' quasi-stationary pdf $q_A(x)$ for {\em any} $\lambda\in(0,\lambda_A)\subset(0,1/8]$. However, while $\varphi(x,\lambda)$ given~\eqref{eq:eigfun-gen-form} is still an eigenfunction of $\mathscr{D}$, it satisfies the square-integrability condition only for $\lambda=\lambda_A$, i.e., $\|\psi(\cdot,\lambda)\|<+\infty$ is false, unless $\lambda=\lambda_A$. The existence such a continuum of quasi-stationary distributions was also previously predicted in~\cite[Section~7.8.2]{Collet+etal:Book2013}; see also~\cite[Corollary 6.19,~p.~144, and Theorem 6.34,~p.~157]{Collet+etal:Book2013}.

%\appendix
%\section{An example appendix}

%+-----------------------------------------------------------------------------------------------+%
\section*{Acknowledgments}
The authors would like to thank Dr.~E.V. Burnaev of the Kharkevich Institute for Information Transmission Problems, Russian Academy of Sciences, Moscow, Russia, and Prof.~A.N. Shiryaev of the Steklov Mathematical Institute, Russian Academy of Sciences, Moscow, Russia, for their interest in and attention to this work.

% Reference list
%+-----------------------------------------------------------------------------------------------+%
\bibliographystyle{siamplain}
\bibliography{main,physics,special-functions,finance,stochastic-processes,differential-equations}

\begin{thebibliography}{10}

\bibitem{Abramowitz+Stegun:Handbook1964}
{\sc M.~Abramowitz and I.~Stegun}, eds., {\em Handbook of Mathematical
  Functions with Formulas, Graphs, and Mathematical Tables}, vol.~55 of Applied
  Mathematics Series, United States National Bureau of Standards, tenth~ed.,
  1964.

\bibitem{Buchholz:Book1969}
{\sc H.~Buchholz}, {\em The Confluent Hypergeometric Function}, vol.~15 of
  Springer Tracts in Natural Philosophy, Springer-Verlag, New York, NY, 1969.
\newblock Translated from German into English by H. Lichtblau and K. Wetzel.

\bibitem{Burnaev+etal:TPA2009}
{\sc E.~V. Burnaev, E.~A. Feinberg, and A.~N. Shiryaev}, {\em On asymptotic
  optimality of the second order in the minimax quickest detection problem of
  drift change for {B}rownian motion}, Theory of Probability and Its
  Applications, 53 (2009), pp.~519--536,
  \url{https://doi.org/10.1137/S0040585X97983791}.

\bibitem{Cattiaux+etal:AP2009}
{\sc P.~Cattiaux, P.~Collet, A.~Lambert, S.~Mart{\'i}nez, S.~M{\'e}l{\'e}ard,
  and J.~S. Mart{\'i}n}, {\em Quasi-stationary distributions and diffusion
  models in population dynamics}, Annals of Probability, 37 (2009),
  pp.~1926--1969, \url{https://doi.org/10.1214/09-AOP451}.

\bibitem{Coddington+Levinson:Book1955}
{\sc E.~A. Coddington and N.~Levinson}, {\em Theory of Ordinary Differential
  Equations}, McGraw-Hill, New York, NY, 1955.

\bibitem{Collet+etal:AP1995}
{\sc P.~Collet, S.~Mart\'{i}nez, and J.~S. Mart\'{i}n}, {\em Asymptotic laws
  for one-dimensional diffusions conditioned to nonabsorption}, Annals of
  Probability, 23 (1995), pp.~1300--1314,
  \url{https://doi.org/10.1214/aop/1176988185}.

\bibitem{Collet+etal:Book2013}
{\sc P.~Collet, S.~Mart\'{i}nez, and J.~S. Mart\'{i}n}, {\em Quasi-Stationary
  Distributions {M}arkov Chains, Diffusions and Dynamical Systems}, Probability
  and Its Applications, Springer, New York, NY, 2013.

\bibitem{Comtet+Monthus:JPhA1996}
{\sc A.~Comtet and C.~Monthus}, {\em Diffusion in a one-dimensional random
  medium and hyperbolic {B}rownian motion}, Journal of Physics A: Mathematical
  and General, 29 (1996), pp.~1331--1345,
  \url{https://doi.org/10.1088/0305-4470/29/7/006}.

\bibitem{DeSchepper+Goovaerts:IME1999}
{\sc A.~{D}e Schepper and M.~J. Goovaerts}, {\em The {GARCH}(1,1)-{M} model:
  {R}esults for the densities of the variance and the mean}, Insurance:
  Mathematics and Economics, 24 (1999), pp.~83--94,
  \url{https://doi.org/10.1016/S0167-6687(98)00040-7}.

\bibitem{DeSchepper+etal:IME1994}
{\sc A.~{D}e Schepper, M.~Teunen, and M.~J. Goovaerts}, {\em An analytical
  inversion of a {L}aplace transform related to annuities certain}, Insurance:
  Mathematics and Economics, 14 (1994), pp.~33--37,
  \url{https://doi.org/10.1016/0167-6687(94)00004-2}.

\bibitem{Donati-Martin+etal:RMI2001}
{\sc C.~Donati-Martin, R.~Ghomrasni, and M.~Yor}, {\em On certain {M}arkov
  processes attached to exponential functionals of {B}rownian motion:
  {A}pplication to {A}sian options}, Revista Matem\'{a}tica Iberoamericana, 17
  (2001), pp.~179--193, \url{https://doi.org/10.4171/RMI/292}.

\bibitem{Dufresne:AAP2001}
{\sc D.~Dufresne}, {\em The integral of geometric {B}rownian motion}, Advances
  in Applied Probability, 33 (2001), pp.~223--241,
  \url{https://doi.org/10.1239/aap/999187905}.

\bibitem{Dunford+Schwartz:Book1963}
{\sc N.~Dunford and J.~T. Schwartz}, {\em Linear Operators. Part II: Spectral
  Theory. Self Adjoint Operators in Hilbert Space}, John Wiley \& Sons, Inc.,
  New York, NY, 1963.

\bibitem{Feinberg+Shiryaev:SD2006}
{\sc E.~A. Feinberg and A.~N. Shiryaev}, {\em Quickest detection of drift
  change for {B}rownian motion in generalized {B}ayesian and minimax settings},
  Statistics \& Decisions, 24 (2006), pp.~445--470,
  \url{https://doi.org/10.1524/stnd.2006.24.4.445}.

\bibitem{Geman+Yor:MF1993}
{\sc H.~Geman and M.~Yor}, {\em {B}essel processes, {A}sian options, and
  perpetuities}, Mathematical Finance, 3 (1993), pp.~349--375,
  \url{https://doi.org/10.1111/j.1467-9965.1993.tb00092.x}.

\bibitem{Gradshteyn+Ryzhik:Book2014}
{\sc I.~S. Gradshteyn and I.~M. Ryzhik}, {\em Table of Integrals, Series, and
  Products}, Academic Press, eighth~ed., 2014.

\bibitem{Ito+McKean:Book1974}
{\sc K.~It\^{o} and H.~P. McKean, Jr.}, {\em Diffusion Processes and Their
  Sample Paths}, Springer, Berlin, Germany, 1974.

\bibitem{Kolmogorov:MA1931}
{\sc A.~Kolmogoroff}, {\em {\"U}ber die analitische {M}ethoden in der
  {W}ahrscheinlichkeitsrechnung}, Mathematische Annalen, 104 (1931),
  pp.~415--458, \url{https://doi.org/10.1007/BF01457949}.
\newblock (in German).

\bibitem{Levitan:Book1950}
{\sc B.~M. Levitan}, {\em Eigenfunction expansions of second-order differential
  equations}, Gostechizdat Moscow--Leningrad, Leningrad, USSR, 1950.
\newblock (in Russian).

\bibitem{Levitan+Sargsjan:Book1975}
{\sc B.~M. Levitan and I.~S. Sargsjan}, {\em Introduction to Spectral Theory:
  {S}elfadjoint Ordinary Differential Operators}, vol.~39 of Translations of
  Mathematical Monographs, American Mathematical Society, Providence, RI, 1975.

\bibitem{Linetsky:OR2004}
{\sc V.~Linetsky}, {\em Spectral expansions for {A}sian (average price)
  options}, Operations Research, 52 (2004), pp.~856--867,
  \url{https://doi.org/10.1287/opre.1040.0113}.

\bibitem{Linetsky:HandbookChapter2007}
{\sc V.~Linetsky}, {\em Spectral methods in derivative pricing}, vol.~15 of
  Handbooks in Operations Research and Management Sciences, North--Holland,
  Netherlands, 2007, ch.~6, pp.~223--299.

\bibitem{Mandl:CMJ1961}
{\sc P.~Mandl}, {\em Spectral theory of semi-groups connected with diffusion
  processes and its application}, Czechoslovak Mathematical Journal, 11 (1961),
  pp.~558--569.

\bibitem{Martinez+SanMartin:JTP2001}
{\sc S.~Mart\'{i}nez and J.~S. Mart\'{i}n}, {\em Rates of decay and
  $h$-processes for one dimensional diffusions conditioned on non-absorption},
  Journal of Theoretical Probability, 14 (2001), pp.~199--212,
  \url{https://doi.org/10.1023/A:1007881317492}.

\bibitem{Martinez+SanMartin:AP2004}
{\sc S.~Mart\'{i}nez and J.~S. Mart\'{i}n}, {\em Classiffication of killed
  one-dimensional diffusions}, Annals of Probability, 32 (2004), pp.~530--552,
  \url{https://doi.org/10.1214/aop/1078415844}.

\bibitem{Milevsky:IME1997}
{\sc M.~A. Milevsky}, {\em The present value of a stochasic perpetuity and the
  {G}amma distribution}, Insurance: Mathematics and Economics, 20 (1997),
  pp.~243--250, \url{https://doi.org/10.1016/S0167-6687(97)00013-9}.

\bibitem{Monthus+Comtet:JPhIF1994}
{\sc C.~Monthus and A.~Comtet}, {\em On the flux distribution in a one
  dimensional disordered system}, Journal de Physique I: France, 4 (1994),
  pp.~635--653, \url{https://doi.org/10.1051/jp1:1994167}.

\bibitem{Peskir:Shiryaev2006}
{\sc G.~Peskir}, {\em On the fundamental solution of the
  {K}olmogorov--{S}hiryaev equation}, in From Stochastic Calculus to
  Mathematical Finance: The {S}hiryaev Festschrift, Y.~Kabanov, R.~Liptser, and
  J.~Stoyanov, eds., Springer, Berlin, 2006, pp.~535--546,
  \url{https://doi.org/10.1007/978-3-540-30788-4_26}.

\bibitem{Pollak+Siegmund:B85}
{\sc M.~Pollak and D.~Siegmund}, {\em A diffusion process and its applications
  to detecting a change in the drift of {B}rownian motion}, Biometrika, 72
  (1985), pp.~267--280, \url{https://doi.org/10.1093/biomet/72.2.267}.

\bibitem{Polunchenko:TPA2017}
{\sc A.~S. Polunchenko}, {\em Asymptotic near-minimaxity of the randomized
  {S}hiryaev--{R}oberts--{P}ollak change-point detection procedure in
  continuous time}, Theory of Probability and Its Applications, 64 (2017),
  pp.~769--786, \url{https://doi.org/10.4213/tvp5142}.

\bibitem{Polunchenko:SA2017}
{\sc A.~S. Polunchenko}, {\em On the quasi-stationary distribution of the
  {S}hiryaev--{R}oberts diffusion}, Sequential Analysis, 36 (2017),
  pp.~126--149, \url{https://doi.org/10.1080/07474946.2016.1275512}.

\bibitem{Polunchenko+Sokolov:MCAP2016}
{\sc A.~S. Polunchenko and G.~Sokolov}, {\em An analytic expression for the
  distribution of the generalized {S}hiryaev--{R}oberts diffusion: The
  {F}ourier spectral expansion approach}, Methodology and Computing in Applied
  Probability, 18 (2016), pp.~1153--1195,
  \url{https://doi.org/10.1007/s11009-016-9478-7}.

\bibitem{Schroder:AAP2003}
{\sc M.~Schr{\"o}der}, {\em On the integral of geometric {B}rownian motion},
  Advances in Applied Probability, 35 (2003), pp.~159--183,
  \url{https://doi.org/10.1239/aap/1046366104}.

\bibitem{Shiryaev:SMD61}
{\sc A.~N. Shiryaev}, {\em The problem of the most rapid detection of a
  disturbance in a stationary process}, Soviet Mathematics---Doklady, 2 (1961),
  pp.~795--799.
\newblock (Translated from Dokl. Akad. Nauk SSSR 138:1039--1042, 1961).

\bibitem{Shiryaev:TPA63}
{\sc A.~N. Shiryaev}, {\em On optimum methods in quickest detection problems},
  Theory of Probability and Its Applications, 8 (1963), pp.~22--46,
  \url{https://doi.org/10.1137/1108002}.

\bibitem{Shiryaev:Bachelier2002}
{\sc A.~N. Shiryaev}, {\em Quickest detection problems in the technical
  analysis of the financial data}, in Mathematical Finance---{B}achelier
  Congress 2000, H.~Geman, D.~Madan, S.~R. Pliska, and T.~Vorst, eds., Springer
  Finance, Springer, Berlin, 2002, pp.~487--521,
  \url{https://doi.org/10.1007/978-3-662-12429-1_22}.

\bibitem{Slater:Book1960}
{\sc L.~J. Slater}, {\em Confluent Hypergeometric Functions}, Cambridge
  University Press, Cambirdge, UK, 1960.

\bibitem{Titchmarsh:Book1962}
{\sc E.~C. Titchmarsh}, {\em Eigenfunction Expansions Associated with
  Second-Order Differential Equations}, Clarendon, Oxford, UK, 1962.

\bibitem{Vanneste+etal:IME1994}
{\sc M.~Vanneste, M.~J. Goovaerts, and E.~Labie}, {\em The distribution of
  annuities}, Insurance: Mathematics and Economics, 15 (1994), pp.~37--48,
  \url{https://doi.org/10.1016/0003-4916(88)90283-7}.

\bibitem{Whittaker:BAMS1904}
{\sc E.~T. Whittaker}, {\em An expression of certain known functions as
  generalized hypergeometric functions}, Bulletin of the American Mathematical
  Society, 10 (1904), pp.~125--134.

\bibitem{Whittaker+Watson:Book1927}
{\sc E.~T. Whittaker and G.~N. Watson}, {\em A Course of Modern Analysis},
  Cambridge University Press, fourth~ed., 1927.

\bibitem{Wong:SPMPE1964}
{\sc E.~Wong}, {\em The construction of a class of stationary {M}arkoff
  processes}, in Stochastic Processes in Mathematical Physics and Engineering,
  R.~Bellman, ed., American Mathematical Society, Providence, RI, 1964,
  pp.~264--276.

\bibitem{Yor:AAP1992}
{\sc M.~Yor}, {\em On some exponential functionals of {B}rownian motion},
  Advances in Applied Probability, 24 (1992), pp.~509--531.

\end{thebibliography}

\end{document}